\documentclass[twoside, 12pt]{amsart}
\usepackage[letterpaper,hmargin=1.25in,vmargin=1.3in]{geometry}
\usepackage{amscd}
\usepackage{verbatim}
\usepackage{color}
\usepackage{comment}

\input xy
\xyoption{all}

\usepackage{amssymb}
\usepackage{hyperref}

\DeclareMathAlphabet{\cat}{OT1}{cmss}{m}{sl}

\newtheorem*{theorem*}{Theorem}
\newtheorem{theorem}{Theorem}[section]

\newtheorem{proposition}[theorem]{Proposition}
\newtheorem{lemma}[theorem]{Lemma}
\newtheorem{corollary}[theorem]{Corollary}

\theoremstyle{definition}
\newtheorem{remark}[theorem]{Remark}

\newcommand{\tens}{\otimes}

\newcommand{\gmu}{\boldsymbol{\mu}}

\newcommand{\op}{^{\mathrm{op}}}

\newcommand{\Ker}{\operatorname{Ker}}

\newcommand{\ind}{\operatorname{\hspace{0.3mm}ind}}

\newcommand{\Inv}{\operatorname{Inv}}

\newcommand{\dec}{\operatorname{dec}}
\newcommand{\disc}{\operatorname{disc}}

\newcommand{\Br}{\operatorname{Br}}

\newcommand{\SB}{\operatorname{SB}}

\newcommand{\gPGL}{\operatorname{\mathbf{PGL}}}
\newcommand{\gSp}{\operatorname{\mathbf{Sp}}}
\newcommand{\gPGSp}{\operatorname{\mathbf{PGSp}}}
\newcommand{\gPGO}{\operatorname{\mathbf{PGO}}}
\newcommand{\gGSp}{\operatorname{\mathbf{GSp}}}
\newcommand{\gSL}{\operatorname{\mathbf{SL}}}
\newcommand{\gO}{\operatorname{\mathbf{O}}}
\newcommand{\gGamma}{\operatorname{\mathbf{\Gamma}}}
\newcommand{\gGL}{\operatorname{\mathbf{GL}}}
\newcommand{\gm}{\operatorname{\mathbb{G}}_m}

\newcommand{\gPGSP}{\operatorname{\mathbf{PGSp}}}
\newcommand{\gSpin}{\operatorname{\mathbf{Spin}}}
\newcommand{\gHSpin}{\operatorname{\mathbf{HSpin}}}
\newcommand{\End}{\operatorname{End}}

\newcommand{\nr}{\operatorname{nr}}

\newcommand{\Nrd}{\operatorname{Nrd}}

\newcommand{\tors}{\operatorname{-\cat{torsors}}}

\newcommand{\Dec}{\operatorname{Dec}}

\newcommand{\norm}{\operatorname{norm}}
\newcommand{\red}{\operatorname{red}}
\newcommand{\Z}{\mathbb{Z}}

\newcommand{\QZ}{\mathop{\mathbb{Q}/\mathbb{Z}}}

\title[Degree three invariants for semisimple groups of types $B$, $C$, and $D$] % colontitle
{Degree three invariants for semisimple groups of types $B$, $C$, and $D$}

\author
[S.~Baek] {Sanghoon Baek}

\address[Sanghoon Baek]{Department of Mathematical Sciences, 
	KAIST,
	291 Daehak-ro, Yuseong-gu,
	Daejeon 305-701,
	Republic of Korea}
\email{sanghoonbaek@kaist.ac.kr}
\urladdr{http://mathsci.kaist.ac.kr/~sbaek/}

\begin{document}
	
\begin{abstract}
We determine the group of reductive cohomological degree $3$ invariants of all split semisimple groups of types $B$, $C$, and $D$. We also present a complete description of the cohomological invariants. As an application, we show that the group of degree $3$ unramified cohomology of the classifying space $BG$ is trivial for all split semisimple groups $G$ of types $B$, $C$, and $D$.
\end{abstract}

\maketitle
%\tableofcontents

\section{Introduction}

A degree $d$ \emph{cohomological invariant} of an algebraic group $G$ defined over a field $F$ is a natural transformation of functors
\[G\tors\to H^{d} \]
on the category of field extensions over $F$, where the functor $G\tors$ takes a field $K/F$ to the set $G\tors(K)$ of isomorphism classes of $G$-torsors over $K$ and the functor $H^{d}$ takes $K$ to the Galois cohomology $H^{d}(K)=H^{d}(K, \QZ(d-1))$. All degree $d$ invariants of $G$ form a group $\Inv^{d}(G)$. This notion was introduced by Serre, and since then it has been intensively studied by Merkurjev and Rost for $d=3$ \cite{GMS, Mer163}.

In this paper, we study degree $3$ cohomological invariants of split semisimple groups of Dynkin types $B$, $C$, and $D$. Thus from now on we shall focus on degree $3$ invariants. Let $G$ be a split reductive group over a field $F$. An invariant in $\Inv^{3}(G)$ is called \emph{normalized} if it vanishes on trivial $G$-torsors. Such invariants form a subgroup $\Inv^{3}(G)_{\norm}$ of $\Inv^{3}(G)$, thus $\Inv^{3}(G)=\Inv^{3}(G)_{\norm}\oplus H^{3}(F)$. A normalized invariant in $\Inv^{3}(G)_{\norm}$ is called \emph{decomposable} if it is given by a cup product of a degree $2$ invariant with a constant invariant of degree $1$. The subgroup of decomposable invariants of degree $3$ is denoted by $\Inv^{3}(G)_{\dec}$. The quotient group $\Inv^{3}(G)_{\norm}/\Inv^{3}(G)_{\dec}$ is called the group of \emph{indecomposable} invariants and is denoted by $\Inv^{3}(G)_{\ind}$. This group has been completely determined for all split simple groups in \cite{GMS}, \cite{Mer163}, \cite{BR} and for some semisimple groups in \cite{Mer161}, \cite{Baek}, \cite{BRZ}, and \cite{Me160}.

Let $G$ be a split semisimple group over $F$. A \emph{strict reductive envelope} of $G$ is a split reductive group $G_{\red}$ over $F$ such that the derived subgroup of $G_{\red}$ is $G$ and the center of $G_{\red}$ is a torus. Then, by \cite[\S10]{Mer162} the restriction map $$\Inv^{3}(G_{\red})_{\ind}\to \Inv^{3}(G)_{\ind}$$ is injective and its image is independent of the choice of a strict reductive envelope  $G_{\red}$. This image is called the subgroup of \emph{reductive indecomposable} invariants of $G$ and is denoted by $\Inv^{3}(G)_{\red}$. Recently, this subgroup has been completely computed for all split simple groups in \cite{LM} and for all split semisimple groups of type $A$ in \cite{Mer161}.

In the present paper, we determine the group of reductive indecomposable invariants of all split semisimple groups of types $B$, $C$, and $D$, which completes the cohomological invariants of classical groups. In particular, if each component of the corresponding root system of type $B$ (respectively, type $C$) has rank at least $2$ (respectively, even rank), then the group of indecomposable invariants is also determined as follows (see Theorem \ref{mainthm}, Theorem \ref{mainthmC}, Theorem \ref{mainthmD}, and Corollary \ref{coromain}):
\begin{theorem}\label{mainintro}
	Let $G$ be an arbitrary split semisimple group of one of the following types$:$ $B$, $C$, and $D$, i.e., $G=(\prod_{i=1}^{m}\gSpin_{2n_{i}+1})/\gmu$ $(n_{i}\geq 1)$, $(\prod_{i=1}^{m}\gSp_{2n_{i}})/\gmu$ $(n_{i}\geq 1)$, and $(\prod_{i=1}^{m}\gSpin_{2n_{i}})/\gmu$ $(n_{i}\geq 3)$ respectively for some central subgroup $\gmu$ and $m\geq 1$. Let $R$ be the subgroup of $Z$ whose quotient is the character group $\gmu^{*}$, where
	\begin{equation*}
		Z:=\bigoplus_{i=1}^{m}Z_{i},\, Z_{i}=\begin{cases} 
			(\Z/2\Z)e_{i} & \text{ if } G \text{ is of type } B \text{ or } C,\\
			(\Z/4\Z)e_{i} & \text{ if } G \text{ is of type } D, n_{i} \text{ odd},\\
			(\Z/2\Z)e_{i,1}\bigoplus (\Z/2\Z)e_{i,2} & \text{ if } G \text{ is of type } D, n_{i} \text{ even,}
		\end{cases}
	\end{equation*}
	denotes the character group of the center of the corresponding simply connected semisimple group.
	
	\noindent$(1)$ Assume that $G$ is of type $B$. Let $l=\dim R$. Then, \[\Inv^{3}(G)_{\red}=(\Z/2\Z)^{l-l_{1}-l_{2}}, \]
		where $l_{1}=\dim \langle e_{i}\in R\,|\, n_{i}\leq 2\rangle$, $l_{2}=\dim\langle e_{i}+e_{j}\in R\,|\, e_{i}, e_{j}\not\in R,\, n_{i}=n_{j}=1\rangle$. In particular, if $n_{i}\geq 2$ for all $1\leq i\leq m$, then
		\[ \Inv^{3}(G)_{\ind}=\Inv^{3}(G)_{\red}=(\Z/2\Z)^{l-l_{1}}.\]

\noindent$(2)$ Assume that $G$ is of type $C$. Let $s$ denote the number of ranks $n_{i}$ divisible by $4$ and $l=\dim \big(R\cap (\bigoplus_{4 \nmid n_{i}}Z_{i})\big)$. Then, 
	\[\Inv^{3}(G)_{\red}=(\Z/2\Z)^{s+l-l_{1}-l_{2}},\]
	where $l_{1}=\dim\langle e_{i}\in R\rangle$ and $l_{2}=\dim\langle e_{i}+e_{j}\in R\,|\, e_{i}, e_{j}\not\in R,\, n_{i}\equiv n_{j}\equiv 1\mod 2\rangle$. In particular, if $n_{i}\equiv 0 \mod 2$ for all $i$, then 
	\[ \Inv^{3}(G)_{\ind}=\Inv^{3}(G)_{\red}=(\Z/2\Z)^{s+l-l_{1}}.\]

\noindent$(3)$ Assume that $G$ is of type $D$. Let \[\bar{R}=\{(\bar{r}_{1},\ldots, \bar{r}_{m})\in \bigoplus_{i=1}^{m}(\Z/2\Z)\bar{e}_{i}\,|\, \sum_{i=1}^{m}r_{i}\in R\}, \text{ where } r_{i}=\begin{cases} 2\bar{r}_{i}e_{i} & \text{ if } n_{i} \text{ odd,}\\ \bar{r}_{i}e_{i,1}+\bar{r}_{i}e_{i,2} & \text{ if } n_{i} \text{ even},\end{cases}\] $R_{1,i}^{}=R\cap Z_{i}$ for odd $n_{i}$, and $R'_{1,i}=R\cap Z_{i}$ for even $n_{i}$. Set
	\begin{align*}
		&R'=\bar{R}\cap \big(\bigoplus_{4\nmid n_{i}, R_{1,i}', R_{1,i}\neq Z_{i}}(\Z/2\Z)\bar{e}_{i}\big) \text{ with } l=\dim R', \, I_{1}=\{i\,|\, Z_{i}=R_{1,i} \text{ or } R_{1,i}', n_{i}\neq 3\},\\
		&I_{2}=\{i \,|\,R_{1,i}'=0,  4|n_{i}\}\cup \{i\,|\, R_{1,i}'=(\Z/2\Z)e_{i,1} \text{ or } (\Z/2\Z)e_{i,2}, n_{i}\geq 6, 4|n_{i}\} \text{ with } s_{i}=|I_{i}|.
	\end{align*}
Then, we have
	\[ \Inv^{3}(G)_{\red}=(\Z/2\Z)^{s_{1}+s_{2}+l-l_{1}-l_{2}},\,\, \text{where}\]
$l_{1}=|\{i \,\,|\,\, 4\nmid n_{i}, R_{1,i}=2Z_{i} \text{ or } R_{1,i}'=(\Z/2\Z)(e_{i,1}+e_{i,2})\}|$, $l_{2}=\dim\langle \bar{e}_{i}+\bar{e}_{j}\,|\, R_{1,i}=R_{1,j}=0,\, 2e_{i}+2e_{j}\in R \rangle$.
	
\end{theorem}

For each type of $B$, $C$, and $D$, our main theorem can be restated as follows (see Propositions \ref{formcor}, \ref{formcorC}, \ref{formcorD}): Assume that $F$ is an algebraically closed field. For type $B$, let $G_{\red}=(\prod_{i=1}^{m}\gGamma_{2n_{i}+1})/\gmu$, where $\gGamma_{2n_{i}+1}$ is the split even Clifford group \cite[\S 23]{KMRT} and let $$R\to \Inv^{3}(G_{\red})_{\norm}$$ be the homomorphism given by $r\mapsto \boldsymbol{\mathrm{e}}_{3}(\phi[r])$, where $\phi[r]$ is the quadratic form defined in Remark \ref{remarkforsimple} and $\boldsymbol{\mathrm{e}}_{3}$ denotes the Arason invariant. Then, this morphism is surjective and its kernel is the subspace $$\langle e_{i},\, e_{j}+e_{k}\in R\,|\, e_{j}, e_{k}\not\in R,\, n_{i}\leq 2,\, n_{j}=n_{k}=1\rangle.$$

For type $C$, let $G_{\red}=(\prod_{i=1}^{m}\gGSp_{2n_{i}})/\gmu$, where $\gGSp_{2n_{i}}$ is the group of symplectic similitudes \cite[\S 12]{KMRT} and let $$\bigoplus_{4\,| n_{i}}(\Z/2\Z) e_{i}\bigoplus\big(R\cap (\bigoplus_{4\nmid n_{i}}(\Z/2\Z) e_{i})\big)\to \Inv^{3}(G_{\red})_{\norm}$$ be the homomorphism given by $e_{i}\mapsto \Delta_{i}$ for $i$ such that $4 | n_{i}$ and $r\mapsto \boldsymbol{\mathrm{e}}_{3}(\phi[r])$ for $r\in R\cap (\bigoplus_{4\nmid n_{i}}(\Z/2\Z) e_{i})$, where $\phi[r]$ is the quadratic form defined in (\ref{phiR}) and $\Delta_{i}$ is the invariant in (\ref{GPTinvarintpro}) induced by the Garibaldi-Parimala-Tignol invariant \cite{GPT}. Then, this morphism is surjective and its kernel is given by
$$\langle e_{i},\, e_{j}+e_{k}\in R\,|\, e_{j}, e_{k}\not\in R,\, n_{j}\equiv n_{k}\equiv 1\mod 2\rangle.$$

For type $D$, let $G_{\red}=(\prod_{i=1}^{m}\operatorname{\mathbf{\Omega}}_{2n_{i}})/\gmu$, where $\operatorname{\mathbf{\Omega}}_{2n_{i}}$ is the extended Clifford group \cite[\S 13]{KMRT} and let $$\bigoplus_{i\in I_{1}\cup I_{2}}(\Z/2\Z) \bar{e}_{i}\bigoplus R'\to \Inv^{3}(G_{\red})_{\norm}$$ be the homomorphism given by $\bar{e}_{i}\mapsto \boldsymbol{\mathrm{e}}_{3, i}$ for $i\in I_{1}$, $\bar{e}_{i}\mapsto \Delta_{i}'$ for $i\in I_{2}$, and $r\mapsto \boldsymbol{\mathrm{e}}_{3}(\phi[r])$ for $r\in R'$, where $\boldsymbol{\mathrm{e}}_{3, i}$ denotes the invariant in (\ref{arasonethree}) induced by the Arason invariant, $\Delta_{i}'$ denotes the invariant in (\ref{merpgo}) given by the invariant of $\gPGO^{+}_{2n_{i}}$ (see \cite[Theorem 4.7]{Mer163}), and $\phi[r]$ is the quadratic form defined in (\ref{phiR}). Then, the morphism is surjective, and its kernel is given by
$$\langle \bar{e}_{i},\, \bar{e}_{j}+\bar{e}_{k}\in R'\,|\, \bar{e}_{j}, \bar{e}_{k}\not\in R',\, n_{j}\equiv n_{k}\equiv 1\mod 2\rangle.$$

Therefore, our main result (Theorem \ref{mainintro}) tells us that for all split semisimple groups of types $B$, $C$, $D$ there are essentially two types of degree three reductive invariants given by the Arason invariant $\boldsymbol{\mathrm{e}}_{3}$ and the Garibaldi-Parimala-Tignol invariant $\Delta_{i}$ (and its analogue $\Delta_{i}'$) and no other invariants exist.

An invariant $\alpha\in \Inv^{3}(G)$ is said to be \emph{unramified} if for any field extension $K/F$ and any element $\eta\in G\tors(K)$, its value $\alpha(\eta)$ is contained in $H^{3}_{\nr}(K)$, where $H^{3}_{\nr}(K)$ denotes the subgroup in $H^{3}(K)$ of all unramified elements defined by
\[H^{3}_{\nr}(K)=\bigcap_{v}\Ker\big(\,\partial_{v}:H^{3}(K)\to H^{2}(F(v))\, \big)  \]
for all discrete valuations $v$ on $K/F$ and their residue homomorphisms $\partial_{v}$. The subgroup of all unramified invariant in $\Inv^{3}(G)$ will be denoted by $\Inv^{3}_{\nr}(G)$. By a theorem of Rost, we have an isomorphism
\begin{equation}\label{Rostisomrphismm}
\Inv^{3}_{\nr}(G)\simeq H^{3}_{\nr}(F(BG)),
\end{equation}
where $BG$ is the classifying space of $G$ (see \cite{Mer162}, \cite{Totaro}).

A generalized version of Noether's problem asks whether the classifying space $BG$ of an algebraic group $G$ is stably rational or retract rational (see \cite{CS}, \cite{Mer17prime}). A way of detecting non-retract rationality is to use unramified cohomology as the following statement: the classifying space $BG$ is not retract rational if there exists a non-constant unramified invariant of degree $d$ for some $d$ \cite{Mer17prime}. In fact, Saltman gave the first counter example over an algebraically closed field to the original Noether's question by providing certain finite groups which have a non-constant unramified invariant of degree $2$  \cite{Sal}. However, the generalized Noether's problem is still open for a connected algebraic group over an algebraically closed field.

In \cite{Bog}, Bogomolov showed that connected groups have no nontrivial degree $2$ unramified invariants, i.e., $\Inv^{2}_{\nr}(G)=0$ for a connected group $G$. In \cite{Sal1} and \cite{Sal2}, Saltman showed that the group $\Inv^{3}_{\nr}(\gPGL_{n})$ is trivial. Recently, Merkurjev has shown that the group $\Inv^{3}_{\nr}(G)$ is trivial if $G$ is a split simple group \cite{Mer162} or a split semisimple group of type $A$ \cite{Mer17} over an algebraically field $F$ of characteristic $0$.

Using the main theorem above we determine the group of unramified invariants of a split semisimple groups of types $B$, $C$, and $D$ (see Theorems \ref{secthm}, \ref{secthmC}, \ref{secthmD}).
\begin{theorem}
	Let $G=(\prod_{i=1}^{m}\gSpin_{2n_{i}+1})/\gmu$ $(n_{i}\geq 1)$ or $(\prod_{i=1}^{m}\gSp_{2n_{i}})/\gmu$ $(n_{i}\geq 1)$ or $(\prod_{i=1}^{m}\gSpin_{2n_{i}})/\gmu$ $(n_{i}\geq 3)$ defined over an algebraically closed field $F$ of characteristic $0$, $m\geq 1$, where $\gmu$ is an arbitrary central subgroup. Then, there are no nontrivial unramified degree $3$ invariants for $G$, i.e., $\Inv^{3}_{\nr}(G)=H^{3}_{\nr}(F(BG))=0$.
\end{theorem}

This paper is organized as follows. In Section $2$ we recall some basic definitions and facts used in the rest of the paper. Sections $3$-$5$ are devoted to the computation of the group of degree $3$ invariants of a split semisimple group $G$ of types $B$, $C$, and $D$. In the last section, we present a description of the degree $3$ invariants of $G$ and a proof of the second main result.

\paragraph{\bf Acknowledgements.} 
I am grateful to Alexander Merkurjev for careful reading and numerous suggestions. I am also grateful to Jean-Pierre Tignol for helpful discussion. This work has been supported by National Research Foundation of Korea (NRF) funded by the Ministry of Science, ICT and Future Planning (2016R1C1B2010037).

\section{Cohomological invariants of degree $3$}\label{CIThree}

In this section we recall some basic notions concerning degree $3$ invariants following \cite{GMS, Mer163}. We shall frequently use these in the following sections.

\subsection{Invariant quadratic forms}\label{subuiqf} Let $\tilde{G}$ be a split semisimple simply connected group of Dynkin type $\mathcal{D}$, i.e., $\tilde{G}=G_{1}\times \cdots \times G_{m}$ for some integer $m\geq 1$, where each $G_{i}$ is a split simple simply connected group of type $\mathcal{D}$. Consider the natural action of the Weyl group $W=W_{1}\times \cdots W_{m}$ of $\tilde{G}$ on the weight lattice $\Lambda=\Lambda_{1}\oplus \cdots \oplus \Lambda_{m}$, where $W_{i}$ (resp. $\Lambda_{i}$) is the Weyl group (resp. the weight lattice) of $G_{i}$. Then, the group of $W$-invariant quadratic forms $S^{2}(\Lambda)^{W}$ on $\Lambda$, denoted by $Q(\tilde{G})$, is a sum of cyclic groups
\[Q(\tilde{G})=\Z q_{1}\oplus \cdots \oplus \Z q_{m},\]
where $q_{i}$ is the normalized Killing form of $G_{i}$ for $1\leq i\leq m$.

Consider an arbitrary split semisimple group $G$ of Dynkin type $\mathcal{D}$, i.e., $G=\tilde{G}/\gmu$, where $\gmu$ is a central subgroup. Let $T$ be a split maximal torus of $G$ and let $T^{*}$ be the group of characters of $T$. Then, the subgroup $Q(G)$ of $W$-invariant quadratic forms on $T^{*}$ is given by
\begin{equation}\label{QGarbi}
Q(G)=S^{2}(T^{*})\cap Q(\tilde{G}).
\end{equation}

\subsection{Degree $3$ invariants} Consider the Chern class map $c_{2}:\Z[T^{*}]\to S^{2}(T^{*})$ defined by $c_{2}(\sum_{i}e^{\lambda_{i}})=\sum_{i<j}\lambda_{i}\lambda_{j}$ \cite[\S 3c]{Mer163}, where $\Z[T^{*}]$ is the group ring of the maximal torus $T$ in Section \ref{subuiqf} and $\lambda_{i}\in T^{*}$. Since $(T^{*})^{W}=0$, the restriction of $c_{2}$ induces a group homomorphism
\begin{equation}\label{ctwodec}
c_{2}:\Z[T^{*}]^{W}\to Q(G)
\end{equation}
We shall write $\Dec(G)$ for the image of $c_{2}$ in (\ref{ctwodec}). For $\lambda\in T^{*}$, we denote by $\rho(\lambda)=\sum_{\chi\in W(\lambda)} e^{\chi}$, where $W(\lambda)$ is the $W$-orbit of $\lambda$. Then, the subgroup $\Dec(G)$ is generated by $c_{2}(\rho(\lambda))=-\tfrac{1}{2}\sum_{\chi\in W(\lambda)} \chi^{2}$. By \cite[Theorem 3.9]{Mer163}, the indecomposable invariants of $G$ is determined by the following exact sequence 
\[0\to \Inv^{3}(G)_{\dec}\to  \Inv^{3}(G)_{\norm}\to Q(G)/\Dec(G)\to 0.\]
In particular, if $F$ is algebraically closed, then we have $\Inv^{3}(G)_{\norm}=Q(G)/\Dec(G)$.

\section{The group $Q(G)$ for semisimple groups $G$ of types $B$, $C$, $D$}\label{QG}

In the present section, we shall compute the group $Q(G)$ for types $B$, $C$, and $D$. 

\subsection{Type $B$}

Let $G=(\prod_{i=1}^{m}\gSpin_{2n_{i}+1})/\gmu$ be an (arbitrary) split semisimple group of type $B$, $m, n_{i}\geq 1$, where $\gmu\simeq (\gmu_{2})^{k}$ is a central subgroup for some $k\geq 0$. Let $T$ be the split maximal torus of $G$ (i.e., $T=(\gm^{\sum n_{i}})/\gmu$) and let 
\begin{equation}\label{relationB}
R=\{r=(r_{1},\ldots, r_{m})\in \bigoplus_{i=1}^{m}(\Z/2\Z) e_{i}\,|\, f_{p}(r)=0, 1\leq p\leq k   \}
\end{equation}
be the subgroup of $\bigoplus_{i=1}^{m}(\Z/2\Z) e_{i}$ whose quotient is the character group $\gmu^{*}$ for some linear polynomials $f_{p}\in \Z/2\Z[t_{1},\ldots, t_{m}]$. We shall simply write $(\Z/2\Z)^{m}$ for $\bigoplus_{i=1}^{m}(\Z/2\Z) e_{i}$. Consider the following commutative diagram of exact sequences
\begin{equation}\label{Tdiagram}
\xymatrix{
	0 \ar@{->}[r] & R \ar@{->}[r] & (\Z/2\Z)^{m} \ar@{->}[r]  & \gmu^{*}  \ar@{=}[d]\ar@{->}[r] & 0\\
	0 \ar@{->}[r] & T^{*} \ar@{->>}[u]\ar@{->}[r]  & \prod_{i=1}^{m}\Z^{n_{i}}\ar@{->>}[u]\ar@{->}[r]& \gmu^{*} \ar@{->}[r] & 0\\
}
\end{equation}
where $T^{*}$ is the corresponding character group and the middle map $\prod_{i=1}^{m}\Z^{n_{i}}\to (\Z/2\Z)^{m}$ is given by 
\begin{equation}\label{middlevB}
\sum a_{i,j}w_{i,j}\mapsto (\bar{a}_{1,n_{1}}, \ldots, \bar{a}_{m,n_{m}}) 
\end{equation}
for $1\leq i\leq m$ and $1\leq j\leq n_{i}$, where $w_{i,j}$ denote the fundamental weights for the $i$th component of the root system of $G$. For the rest of this subsection, we simply write $a_{i}$ and $w_{i}$ for $a_{i,n_{i}}$ and $w_{i,n_{i}}$, respectively. Then, it follows from (\ref{Tdiagram}) that
\begin{equation*}
T^{*}=\{\sum a_{i,j}w_{i,j}\,|\, f_{p}(a_{1},\ldots, a_{m})\equiv 0 \mod 2    \}.
\end{equation*}

Let $I=\{1,\ldots, m\}$ and let $I_{1}=\{i\in I\,|\, f_{p}(e_{i})=0,\, 1\leq p\leq k \}$, where $\{e_{1}, \ldots, e_{m}\}$ denotes the standard basis of $\Z^{m}$. We write the relations $f_{p}(a_{1},\ldots, a_{m})\equiv 0 \mod 2$ as
\begin{equation}\label{choicemtxB}
(a_{i_{1}},\ldots, a_{i_{k}})^{T}=B\cdot (a_{j_{1}},\ldots, a_{j_{l}})^{T}+(2c_{1},\ldots, 2c_{k})^{T}
\end{equation}
for some distinct $i_{1},\ldots, i_{k}, j_{1},\ldots, j_{l}$ such that $\{i_{1},\ldots, i_{k}, j_{1},\ldots, j_{l}\}=I\backslash I_{1}$ and some $k\times l$ binary matrix $B=(b_{ij})$ (i.e., $b_{ij}=0$ or $1$) with $c_{p}\in \Z$. Then, we have
\[\sum a_{i,j}w_{i,j}=\!\sum_{1\leq i\leq m, 1\leq j\leq n_{i}-1} a_{i,j}w_{i,j}+\sum_{i\in I_{1}}a_{i}w_{i}+\sum_{p=1}^{k} 2c_{p}w_{i_{p}}+\sum_{s=1}^{l} a_{j_{s}}(w_{j_{s}}+g_{s})  \]
where $g_{s}=(w_{i_{1}},\ldots, w_{i_{k}})\cdot B_{s}$ and $B_{s}$ is the $s$-th column of $B$, thus we obtain the following $\Z$-basis of $T^{*}$:
\begin{equation}\label{Tbasis}
\{w_{i,j}\}_{1\leq i\leq m, 1\leq j\leq n_{i}-1}\cup \{w_{i}\}_{i\in I_{1}}\cup \{2w_{i_{p}}\}_{1\leq p\leq k}\cup \{w_{j_{s}}+g_{s}\}_{1\leq s\leq l}. 
\end{equation}

Let $v_{p}=2w_{i_{p}}$ and $h_{p}(t_{1},\ldots, t_{l})=b_{p1}t_{1}+\cdots +b_{pl}t_{l}\in \Z/2\Z[t_{1},\ldots, t_{l}]$ for $1\leq p\leq k$. Since the group $Q(\tilde{G})$ is generated by the normalized Killing forms
\[q_{i}=\begin{cases}
2w_{i}^{2}-2w_{i,n_{i}-1}w_{i}-\sum_{j=1}^{n_{i}-2}w_{i,j}w_{i,j+1}+\sum_{j=1}^{n_{i}-1}w_{i,j}^{2} & \text{ if } n_{i}\geq 1,\\
w_{i}^{2} & \text{ if } n_{i}=1
\end{cases}\]
for all $1\leq i\leq m$, any element of $Q(G)$ is of the form $q=\sum_{i=1}^{m}d_{i}q_{i}$ for some $d_{i}\in \Z$. Therefore, with respect to the basis (\ref{Tbasis}) we have
\[q=q'+\tfrac{1}{4}\sum_{p=1}^{k} v_{p}^{2}[\delta_{i_{p}}d_{i_{p}}+h_{p}(\delta_{j_{1}}d_{j_{1}}, \ldots, \delta_{j_{l}}d_{j_{l}})]+\tfrac{1}{2}\sum_{1\leq i<j\leq k} v_{i}v_{j}h_{i}(\delta_{j_{1}}d_{j_{1}}b_{j_{1}}, \ldots, \delta_{j_{l}}d_{j_{l}}b_{j_{l}}) \]
for some quadratic form $q'$ with integer coefficients, where 
\[\delta_{i}=\begin{cases}
2 & \text{ if } n_{i}\geq 2 \text{ with }i\in I\backslash I_{1},\\
1 & \text{ if } n_{i}=1\text{ with }i\in I\backslash I_{1}.
\end{cases}
\]
Hence, by (\ref{QGarbi}) we obtain $q=\sum_{i=1}^{m}d_{i}q_{i}\in Q(G)$ if and only if
\begin{equation}\label{firsteq}
\delta_{i_{p}}d_{i_{p}}+h_{p}(\delta_{j_{1}}d_{j_{1}}, \ldots, \delta_{j_{l}}d_{j_{l}})\equiv 0 \mod 4
\end{equation}
and
\begin{equation}\label{secondeq}
h_{p}(\delta_{j_{1}}d_{j_{1}}b_{j_{1}}, \ldots, \delta_{j_{l}}d_{j_{l}}b_{j_{l}})\equiv 0 \mod 2
\end{equation}
for all $1\leq p\leq k$. In particular, since two systems of equations $\{f_{p}(t_{1},\ldots, t_{m})\}$ and $\{t_{i_{p}}+h_{p}(t_{j_{1}},\ldots, t_{j_{l}})\}$ are equivalent we replace the condition (\ref{firsteq}) by
\begin{equation}\label{firsteqprime}
f_{p}(\delta_{1}d_{1},\ldots, \delta_{m}d_{m})\equiv 0 \mod 4,
\end{equation}
where we set $\delta_{i}=2$ for $i\in I_{1}$.  

Equivalently, we can compute $Q(G)$ with respect to a basis of $R$ as follows. Let 
\begin{equation}\label{RoneRtwo}
R_{1}=\langle e_{i}\,|\, e_{i}\in R\rangle \text{ and } R_{2}=\langle e_{i}+e_{j}\,|\, e_{i}+e_{j}\in R, e_{i}, e_{j}\not\in R_{1}\rangle
\end{equation}
be the subspaces of $R$. We first choose $\{w_{i}\}_{i\in I_{1}}$ as a part of basis of $T^{*}$. Then, for the remaining part of a basis of $T^{*}$ we write a given basis of $R$ as
\begin{equation}\label{Rtwo}
(e_{j_{1}},\ldots, e_{j_{l}})^{T}=C(e_{i_{1}},\ldots, e_{i_{k}})^{T}
\end{equation}
for some $i_{1},\ldots, i_{k}, j_{1},\ldots, j_{l}$ with $\{i_{1},\ldots, i_{k}, j_{1},\ldots, j_{l}\}=I\backslash I_{1}$ and some $l\times k$ binary matrix $C$ such that all basis elements of the form $e_{i}+e_{j}$ in $R_{2}$ is a part of (\ref{Rtwo}). Then, we have the same $\Z$-basis of $T^{*}$ as in (\ref{Tbasis}) by replacing $g_{s}$ in (\ref{Tbasis}) with $g_{s}=C_{s}\cdot (w_{i_{1}}, \ldots, w_{i_{k}})$, where $C_{s}$ is the $s$-th row of $C$. The rest of the computation is the same as in the previous one. 

In particular, if either $R=R_{1}\oplus R_{2}$ or $n_{i}\geq 2$ for all $1\leq i\leq m$, then the condition (\ref{secondeq}) becomes trivial, thus

\begin{proposition}\label{QGprop}
Let $G=(\prod_{i=1}^{m}\gSpin_{2n_{i}+1})/\gmu$, $m, n_{i}\geq 1$, where $\gmu\simeq (\gmu_{2})^{k}$ is a central subgroup for some $k\geq 0$. Let $R=\{r\in (\Z/2\Z)^{m}\,|\, f_{p}(r)=0, 1\leq p\leq k \}$ be the subgroup of $(\gmu_{2}^{m})^{*}$ whose quotient is the character group $\gmu^{*}$ for some linear polynomials $f_{j}\in \Z/2\Z[t_{1},\ldots, t_{m}]$. Assume that either $n_{i}\geq 2$ for all $i$ or $R=R_{1}\oplus R_{2}$, where $R_{1}$ and $R_{2}$ are the subgroups of $R$ defined in (\ref{RoneRtwo}). Then, we have
\[Q(G)=\{\sum_{i=1}^{m}d_{i}q_{i}\,|\,f_{p}(\delta_{1}d_{1},\ldots, \delta_{m}d_{m})\equiv 0 \mod 4\}.\] 
\end{proposition}

\subsection{Type $C$}

Let $G=(\prod_{i=1}^{m}\gSp_{2n_{i}})/\gmu$ be a split semisimple group of type $C$, where $m, n_{i}\geq 1$ and $\gmu\simeq (\gmu_{2})^{k}$ is a central subgroup for some $k\geq 0$. Let $T$ be the split maximal torus of $G$ and let $R$ be the subgroup of $(\Z/2\Z)^{m}$ as in (\ref{relationB}). Then, we have the same commutative diagram (\ref{Tdiagram}), replacing the middle vertical map (\ref{middlevB}) by
\[\sum a_{i,j}e_{i,j}\mapsto (\sum_{j=1}^{n_{1}} \bar{a}_{1,j}, \ldots, \sum_{j=1}^{n_{m}} \bar{a}_{m,j}),  \]
where $e_{i,j}$ denote the standard basis for the $i$th component of $\prod_{i=1}^{m}\Z^{n_{i}}$. Then, by (\ref{Tdiagram}) we have
\begin{equation}\label{charTC}
T^{*}=\{\sum a_{i,j}e_{i,j}\,|\, f_{p}(\sum_{j=1}^{n_{1}} a_{1,j},\ldots, \sum_{j=1}^{n_{m}} a_{m,j})\equiv 0 \mod 2    \}.
\end{equation}

We simply write $e_{i}$ for $e_{i,1}$. Let $e'_{i,j}=e_{i,j}-e_{i}$ for all $1\leq i\leq m$ and $2\leq j\leq n_{i}$ and let $a_{i}=\sum_{j=1}^{n_{i}} a_{i,j}$. Then, we apply the same argument as in type $B$ so that we have the following $\Z$-basis of $T^{*}$
\begin{equation}\label{TbasisC}
\{e'_{i,j}\}_{1\leq i\leq m,\, 2\leq j\leq n_{i}}\cup \{e_{i}\}_{i\in I_{1}}\cup \{2e_{i_{p}}\}_{1\leq p\leq k}\cup \{e_{j_{s}}+g_{s}\}_{1\leq s\leq l}, 
\end{equation}
where $B$ is the binary matrix as in (\ref{choicemtxB}) and $g_{s}=(e_{i_{1}},\ldots, e_{i_{k}})\cdot B_{s}$.

Let $v_{p}=2e_{i_{p}}$ and let $h_{p}$ be the polynomial defined as in type $B$. Since the normalized Killing forms are given by 
\[q_{i}=e_{i,1}^{2}+\cdots+e_{i,n_{i}}^{2},  \]
for any $q\in Q(G)$ there exist $d_{i}\in \Z$ such that $q=\sum_{i=1}^{m}d_{i}q_{i}$, thus with respect to the basis (\ref{TbasisC}) we have
\[q=q'+\tfrac{1}{4}\sum_{p=1}^{k} v_{p}^{2}[n_{i_{p}}d_{i_{p}}+h_{p}(n_{j_{1}}d_{j_{1}}, \ldots, n_{j_{l}}d_{j_{l}})]+\tfrac{1}{2}\sum_{1\leq i<j\leq k} v_{i}v_{j}h_{i}(n_{j_{1}}d_{j_{1}}b_{j_{1}}, \ldots, n_{j_{l}}d_{j_{l}}b_{j_{l}}) \]
for some quadratic form $q'$ with integer coefficients. Therefore, by the same argument as in type $B$ we have $q=\sum_{i=1}^{m}d_{i}q_{i}\in Q(G)$ if and only if
\begin{equation}\label{firsteqtypeC}
h_{p}(n_{j_{1}}d_{j_{1}}b_{j_{1}}, \ldots, n_{j_{l}}d_{j_{l}}b_{j_{l}})\equiv 0\mod 2 \,\text{ and } f_{p}(\delta_{1}n_{1}d_{1},\ldots, \delta_{m}n_{m}d_{m})\equiv 0 \mod 4
\end{equation}
for all $1\leq p\leq k$, where 
\[\delta_{i}=\begin{cases}
1 & \text{ if } i\in I\backslash I_{1},\\
\tfrac{2}{\,\,n_{i}} & \text{ if } i\in I_{1}.
\end{cases}
\]

Similar to the case of type $B$, if $R=R_{1}\oplus R_{2}$ or $n_{i}$ is even for all $1\leq i\leq m$, then the first condition in (\ref{firsteqtypeC}) becomes obvious, thus
\begin{proposition}\label{QGpropC}
	Let $G=(\prod_{i=1}^{m}\gSp_{2n_{i}})/\gmu$, $m, n_{i}\geq 1$, where $\gmu\simeq (\gmu_{2})^{k}$ is a central subgroup for some $k\geq 0$. Let $R$, $R_{1}$, and $R_{2}$ be the groups as in $($\ref{relationB}$)$ and $($\ref{RoneRtwo}$)$. Assume that either $n_{i}$ is even for all $i$ or $R=R_{1}\oplus R_{2}$. Then, we have
	\[Q(G)=\{\sum_{i=1}^{m}d_{i}q_{i}\,|\,f_{p}(\delta_{1}n_{1}d_{1},\ldots, \delta_{m}n_{m}d_{m})\equiv 0 \mod 4\}.\] 
\end{proposition}

\subsection{Type $D$}\label{subD}

Let $G=(\prod_{i=1}^{m}\gSpin_{2n_{i}})/\gmu$ be a split semisimple group of type $D$, $m\geq 1$, $n_{i}\geq 3$, where $\gmu\simeq (\gmu_{2})^{k_{1}}\times (\gmu_{4})^{k_{2}}$ is a subgroup of the center $Z(\prod_{i=1}^{m}\gSpin_{2n_{i}})$ for some $k_{1}, k_{2}\geq 0$.  We shall denote the character group $Z(\prod_{i=1}^{m}\gSpin_{2n_{i}})^{*}$ by
\begin{equation}\label{centerD}
Z:=\bigoplus_{i=1}^{m}Z_{i}, \text{ where } Z_{i}=\begin{cases} 
(\Z/4\Z)e_{i} & \text{ if } n_{i} \text{ odd},\\
(\Z/2\Z)e_{i,1}\bigoplus (\Z/2\Z)e_{i,2} & \text{ if } n_{i} \text{ even}.
\end{cases}
\end{equation}

Let $T$ be the split maximal torus of $G$ and let
\[R=\{ r\in Z\,|\, f_{p}(r)=0, 1\leq p\leq k\}\]
be the subgroup of $Z$ such that $\gmu^{*}\simeq Z/R$ for some linear polynomials $f_{1},\ldots, f_{k}\in \Z/4\Z[T_{1},\ldots, T_{m}]$ with $k=k_{1}+k_{2}$, where $T_{i}$ denotes a $2$-tuple $(t_{i1}, t_{i2})$ of variables (resp. a variable $t_{i}$) if $n_{i}$ is even (resp. odd) and the coefficients of $t_{i1}$ and $t_{i2}$ in $f_{p}$ are either $0$ or $2$. Then, we have the same diagram (\ref{Tdiagram}), replacing the middle vertical map (\ref{middlevB}) by $\prod_{i=1}^{m}\Z^{n_{i}}\to Z$,
\begin{equation}\label{defaii}
\sum_{j=1}^{n_{i}} a_{i,j}w_{i,j}\mapsto A_{i}:=\begin{cases} \big(\overline{a_{i,n_{i}-1}-a_{i, n_{i}}+2S_{i}}\big)e_{i} & \text{ if } n_{i} \text{ odd},\\ \big(\overline{a_{i,n_{i}-1}+S_{i}}\big)e_{i1}+\big(\overline{a_{i,n_{i}}+S_{i}}\big)e_{i2}  & \text{ if } n_{i} \text{ even},\end{cases} 
\end{equation}
where $S_{i}=\sum_{j=1}^{[(n_{i}-1)/2]}a_{i,2j-1}$ and $w_{i,j}$ denote the fundamental weights for the $i$th component of the root system of $G$. Therefore, by (\ref{Tdiagram}) we have
\begin{equation}\label{charTD}
T^{*}=\{\sum a_{i,j}w_{i,j}\,|\, f_{p}(\sum_{i=1}^{m} A_{i})=0, 1\leq p\leq k\}.
\end{equation}

Let $I_{1}'=\{i\in I\,|\, f_{p}(e_{i})=0 \text{ or } f_{p}(e_{i,1})=f_{p}(e_{i,2})=0 \text{ for all } 1\leq p\leq k\}$ and $I'=I\backslash I_{1}'$. In view of the argument in the case of type $B$ we may assume that each relation $f_{p}(\sum_{i=1}^{m} A_{i})=0$ can be written as \[\delta_{p}a_{p}=b_{p}+4c_{p}, \text{ where } b_{p}=\begin{cases} \delta_{p}a_{p}+f_{p}(\sum_{i=1}^{m} A_{i}) & \text{ if }a_{p}=a_{i, n_{i}} \text{ with odd } n_{i}, \\ \delta_{p}a_{p}-f_{p}(\sum_{i=1}^{m} A_{i}) & \text{ otherwise,}\end{cases}\] for some distinct $a_{p}\in \{a_{i, n_{i}-1}, a_{i, n_{i}}\,|\, i\in I'\}$ with  $\delta_{p}\in \{1, 2\}$ and  $c_{p}\in \Z$ such that the terms $a_{1},\ldots, a_{k}$ do not appear in $b_{1},\ldots, b_{k}$ and each coefficient of $a_{i,l}$ in $b_{p}$ is divisible by $\delta_{p}$.

Let $W_{1}=\{w_{i, 2j-1}\,|\, i\in I', 1\leq j\leq [(n_{i}-1)/2]\}\cup \{w_{i, n_{i}-1}, w_{i, n_{i}}\,|\, i\in I'\}$. We simply write $w_{p}\in W_{1}$ for $w_{i, n_{i}-1}$ (resp. $w_{i, n_{i}}$) if $a_p=a_{i, n_{i}-1}$ (resp. $a_{p}=a_{i, n_{i}}$). Set 
\[g_{i, l}=s_{1}(i, l)w_{1}+\cdots +s_{k}(i, l)w_{k} \text{ and } W'=W_{1}\backslash \{w_{1},\ldots, w_{k}\},\]
where  $s_{p}(i,l)$ denotes the coefficient of $a_{i,l}$ in $b_{p}/\delta_{p}$. Then, we obtain the following $\Z$-basis of $T^{*}$:
\begin{equation}\label{TDbasis}
\{w_{i,j}\}_{i\in I_{1}, \forall j}\cup \{w_{i, 2j}\}_{i\in I', 1\leq j\leq [\tfrac{n_{i}-2}{2}]}\cup \{\tfrac{4\, }{\delta_{p}}w_{p}\}_{1\leq p\leq k}\cup \{w_{i,l}+g_{i,l}\}_{w_{i,l}\in W'}. 
\end{equation}

Let $v_{p}=\tfrac{4\,}{\delta_{p}}w_{p}$ and $v_{i,l}=w_{i,l}+g_{i,l}$. Assume that for each $p$, $w_{p}$ is a fundamental weight for the $i_{p}$-th component of the root system of $G$. As the normalized Killing forms are given by
\[q_{i}=\big(\sum_{j=1}^{n_{i}} w_{i,j}^{2}\big)-\big(w_{i,n_{i}-2}w_{i,n_{i}}+\sum_{j=1}^{n_{i}-2}w_{i,j}w_{i,j+1}\big),\]
for any $q\in Q(G)$ there exist $d_{i}\in \Z$ such that $q=\sum_{i=1}^{m}d_{i}q_{i}$. Hence, with respect to the basis (\ref{TDbasis}) we obtain
\begin{align*}
q=&q'+\tfrac{1}{16}\sum_{p=1}^{k}v_{p}^{2}\delta_{p}^{2}[d_{i_{p}}+\sum_{w_{i,l}\in W'}d_{i}s_{p}(i,l)^{2}]+\tfrac{1}{8}\sum_{1\leq p<u\leq k}v_{p}v_{u}\delta_{p}\delta_{u}[\sum_{w_{i,l}\in W'}d_{i}s_{p}(i,l)s_{u}(i,l)]\\
&-\tfrac{1}{2}\sum_{p=1}^{k}v_{p}\delta_{p}[\sum_{w_{i,l}\in W'}v_{il}d_{i}s_{p}(i,l)]  
\end{align*}
for some quadratic form $q'$ with integer coefficients. Hence, $q=\sum_{i=1}^{m}d_{i}q_{i}\in Q(G)$ if and only if
\begin{align}
&\delta_{p}^{2}[d_{i_{p}}+\sum_{w_{i,l}\in W'}d_{i}s_{p}(i,l)^{2}]\equiv 0 \mod 16, \sum_{w_{i,l}\in W'}d_{i}\delta_{p}\delta_{u}s_{p}(i,l)s_{u}(i,l)\equiv 0 \mod 8, \label{firsteqtypeD}\\  &\text{ and } d_{i}\delta_{p}s_{p}(i,l)\equiv 0 \mod 2 \label{firsteqtypeDprime}
\end{align}
for all $1\leq p\leq k$, $1\leq p<u\leq k$, and all $(i, l)$ such that $w_{i,l}\in W'$.

Let $c_{i,1}(p)$, $c_{i,2}(p)$, $c_{i}(p)$ denote the coefficients of $t_{i,1}, t_{i,2}, t_{i}$ in $f_{p}$, respectively. Note that $c_{i,1}(p)$ and $c_{i,2}(p)$ are either $0$ or $2$. Since
\begin{equation*}
\delta_{p}^{2}+\sum_{l}\delta_{p}^{2}s_{p}(i_{p},l)^{2}=\sum_{l}\delta_{p}^{2}s_{p}(i,l)^{2}=\begin{cases} 8 & \text{ if } c_{i}(p)=2 \text{ or } c_{i,1}(p)+c_{i,2}(p)=4,\\
	2n_{i} & \text{ if } c_{i}(p)=\pm 1 \text{ or } c_{i,1}(p)+c_{i,2}(p)=2\end{cases}
\end{equation*}
for all $p$ and $i\neq i_{p}$, where the sums range over all $l$ such that $w_{i,l}\in W'$, the first equation in (\ref{firsteqtypeD}) is equivalent to the following equation
\begin{equation}\label{equifirstD}
f_{p}(T_{1},\ldots, T_{m})\equiv 0 \mod 8, \text{ where } t_{i}=\begin{cases} \pm n_{i}d_{i} & \text{ if } c_{i}(p)=\pm 1,\\ 2d_{i} & \text{ if } c_{i}(p)=2,\end{cases} 
\end{equation}
\begin{equation*}
t_{i,1}=\begin{cases} \tfrac{n_{i}d_{i}}{2} &  \text{ if } c_{i,1}(p)=2, c_{i,2}(p)=0,\\ d_{i} & \text{ if }  c_{i,1}(p)+c_{i,2}(p)=4,\end{cases} \text{ and }\, t_{i,2}=\begin{cases} \tfrac{n_{i}d_{i}}{2} &  \text{ if } c_{i,1}(p)=0, c_{i,2}(p)=2,\\ d_{i} & \text{ if }  c_{i,1}(p)+c_{i,2}(p)=4\end{cases}
\end{equation*}
for all $i\in I'$ and we set $t_{i}=4d_{i}$, $t_{i1}=t_{i2}=2d_{i}$ for all $i\in I_{1}'$. Since we have
\begin{equation*}
\sum_{l}s_{p}(i,l)s_{u}(i,l)\equiv \begin{cases} \pm2n_{i} \mod 8 & \text{ if } c_{i}(p)c_{i}(u)\equiv \pm 1 \mod 4,\\  4\,\,\,\,\,\,\,\,\, \mod 8 & \text{  if } c_{i}(p)c_{i}(u)\equiv 2 \mod 4,\\   0\,\,\,\,\,\,\,\,\, \mod 8 & \text{ otherwise}\end{cases}
\end{equation*}
for all $1\leq p<u\leq k$ such that $\delta_{p}=\delta_{u}=1$, where the sum ranges over all $l$ such that $w_{i,l}\in W'$, the second equation in (\ref{firsteqtypeD}) is equivalent to 
\begin{equation}\label{secondconditionre}
\sum_{\{i\in I' | c_{i}(p)c_{i}(u)\equiv \pm 1 \mod 4 \}} 2d_{i}+\sum_{\{i\in I' | c_{i}(p)c_{i}(u)\equiv 2 \mod 4 \}}4d_{i}\equiv 0 \mod 8 
\end{equation}
if $\delta_{p}=\delta_{u}=1$ and
\[4\sum_{i\in I''}d_{i}\equiv 0 \mod 8\]
for some subset $I''$ of $I'$ otherwise.

\section{The subgroup $\Dec(G)$ for semisimple groups $G$ of types $B$, $C$, $D$}

In this section we will compute the subgroup $\Dec(G)$ of decomposable elements of $G$ for types $B$, $C$, and $D$. In this section we shall denote by $T$ and $T^{*}$ the maximal split torus of $G$ and its character group, respectively and we denote by $\Lambda$ and $\Lambda_{r}$ the weight lattice and the root lattice of $G$, respectively. The Weyl group of $G$ will be denoted by $W$.

\subsection{Type $B$}
Consider a split semisimple group $G=(\prod_{i=1}^{m}\gSpin_{2n_{i}+1})/\gmu$ of type $B$, where $m, n_{i}\geq 1$ and $\gmu\simeq (\gmu_{2})^{k}$ is a central subgroup for some $k\geq 0$. Let $I=\{1,\ldots, m\}$.

We first consider the case where $G$ is simply connected (i.e. $G=\tilde{G}$), equivalently $k=0$. Since 
\begin{equation}\label{decsum}
\Dec(G_{1}\times G_{2})=\Dec(G_{1})\times \Dec(G_{2})
\end{equation}
for any two semisimple groups $G_{1}$ and $G_{2}$, it suffices to compute $\Dec(\gSpin_{2n+1})$. Observe that $\Dec(\gSpin_{3})=\Z q$ as $c_{2}(\rho(w_{1}))=-q$ and $\Dec(\gSpin_{5})=\Z q$ as $c_{2}(\rho(w_{2}))=-q$. Similarly, $c_{2}(\rho(w_{1}))=-2q\in \Dec(\gSpin_{2n+1})$ for any $n\geq 2$. As the Weyl group of $\gSpin_{2n+1}$ contains a normal subgroup $(\Z/2\Z)^{n}$ generated by sign switching, we see that $2 \,|\,c_{2}(\rho(\lambda))$ for any $\lambda\in \Lambda$ (c.f. \cite[Part II, \S
13]{GMS}), thus $\Dec(\gSpin_{2n+1})=2\Z q$. Therefore, 
\begin{equation}\label{decsimply}
\Dec(\tilde{G})=\delta_{1}'\Z q_{1}\oplus \cdots \oplus\delta_{m}'\Z q_{m}, \text{ where } \delta_{i}'=\begin{cases}
2 & \text{ if } n_{i}\geq 3,\\
1 & \text{ if } n_{i}=1, 2.
\end{cases}
\end{equation}

Now we assume that $G$ is adjoint (i.e. $G=\bar{G}$), equivalently, $k=m$. Then, $\Dec(\gO^{+}_{3})=4\Z q$ as $c_{2}(\rho(2w_{1}))=-4q$. Similarly, by the same argument as in the simply connected case, we see that $\Dec(\gO^{+}_{2n+1})=2\Z q$ for $n\geq 2$ (see \cite[Theorem 4.5]{Mer163}). Hence, 
\begin{equation}\label{adjoint}
\Dec(\bar{G})=\delta_{1}''\Z q_{1}\oplus \cdots \oplus \delta_{m}''\Z q_{m}, \text{ where } \delta_{i}''=\begin{cases}
2 & \text{ if } n_{i}\geq 2,\\
4 & \text{ if } n_{i}=1.
\end{cases}
\end{equation}

In general, we show that the subgroup $\Dec(G)$ is determined by certain subgroups of $R$ introduced in Section \ref{QG}.
\begin{proposition}\label{decG}
Let $G=(\prod_{i=1}^{m}\gSpin_{2n_{i}+1})/\gmu$, $m, n_{i}\geq 1$, where $\gmu$ is a central subgroup. Let $R$ be the subgroup of $(\gmu_{2}^{m})^{*}=(\Z/2\Z)^{m}$ such that $\gmu^{*}=(\gmu_{2}^{m})^{*}/R$. Let $R_{1}'=\langle e_{i}\in R\,|\, n_{i}\leq 2\rangle$ and $R_{2}'=\langle e_{i}+e_{j}\in R\,|\, e_{i}, e_{j}\not\in R,\, n_{i}=n_{j}=1\rangle$ be two subspaces of $R$ with $\dim R_{1}'=l_{1}$ and $\dim R_{2}'=l_{2}$. Then,
\begin{equation}\label{propdec}
\Dec(G)=(\bigoplus_{e_{i}\in R_{1}'}\Z q_{i})\oplus (\bigoplus_{n_{i}\geq 2,\, e_{i}\not\in R_{1}'} 2\Z q_{i})\oplus (\bigoplus_{r=1}^{l_{2}} 2\Z q'_{r})\oplus (\bigoplus_{s=1}^{l_{3}} 4\Z q''_{s}),        
\end{equation}
where $l_{3}=m-l_{1}-l_{2}-|\{i \,|\, n_{i}\geq 2,\, e_{i}\not\in R_{1}'\}|$ and $q_{r}'$ $($resp. $q_{s}'')$ is of the form $q_{i}+q_{j}$  $($resp. $q_{i})$ for some $i, j$ such that $\langle q_{r}', q_{s}''\,|\, 1\leq r\leq l_{2}, 1\leq s\leq l_{3}\rangle=\langle q_{i}\,|\, n_{i}=1, e_{i}\not\in R_{1}'\rangle$ over $\Z$.
\end{proposition}
\begin{proof}
It follows from (\ref{decsimply}) and (\ref{adjoint}) that we have
\begin{equation}\label{decbetween}
\delta_{1}''\Z q_{1}\oplus \cdots \oplus \delta_{m}''\Z q_{m}\subseteq \Dec(G)\subseteq \delta_{1}'\Z q_{1}\oplus \cdots \oplus\delta_{m}'\Z q_{m}.\end{equation}
By a simple computation, we obtain
\begin{equation}\label{decctwo}
-c_{2}(\rho(\chi))=\begin{cases}
a_{i}^{2}q_{i} & \text{ if } \chi=a_{i}w_{i,1},\, n_{i}=1,\\
2(a_{i}^{2}q_{i}+a_{j}^{2}q_{j}) & \text{ if } \chi=a_{i}w_{i,1}+a_{j}w_{j,1},\, n_{i}=n_{j}=1\\
\end{cases}
\end{equation}
for any nonzero integers $a_{i}, a_{j}$ and 
\begin{equation}\label{decctwoplus}
-c_{2}(\rho(\chi))=(2a_{i,1}^{2}+a_{i,2}^{2}+2a_{i,1}a_{i,2})q_{i}\,\,\text{ if } \chi=a_{i,1}w_{i,1}+a_{i,2}w_{i, 2},\, n_{i}=2
\end{equation}
for any integers $a_{i,1}, a_{i,2} $. Let us denote the right hand side of equation (\ref{propdec}) by $D$. We write $D=\bigoplus D_{u}$, where $D_{u}$ denotes u-th direct summand of $D$ for $1\leq u\leq 4$. First, we show that $D\subseteq \Dec(G)$. If $e_{i}\in R_{1}'$, then by (\ref{Tbasis}) we have $w_{i,1}$, $w_{i,2}\in T^{*}$, thus by (\ref{decctwo}) and (\ref{decctwoplus}) $D_{1}\subseteq \Dec(G)$. Similarly, if $e_{i}+e_{j}\in R$, then by (\ref{Tbasis}), $w_{i,1}+w_{j,1}\in T^{*}$, thus by (\ref{decctwo}) $D_{3}\subseteq \Dec(G)$. Finally, it follows from (\ref{decbetween}) that $D_{2}\oplus D_{4}\subseteq \Dec(G)$.

On the other hand, a character $\lambda$ in the weight lattice $\Lambda=\bigoplus_{i=1}^{m} \Lambda_{i}$ of $G$ can be written as
\begin{equation}\label{lambbbdaexp}
\lambda=\lambda_{i_{1}}+\cdots +\lambda_{i_{t}}=\sum_{j\in J}\lambda_{i_{j}}+\sum_{j\in K}\lambda_{i_{j}}
\end{equation}
for some nonzero characters $\lambda_{i_{j}}\in \Lambda_{i_{j}}$ and some subsets $J=\{1\leq j\leq t\,|\, n_{i_{j}}=1\}$ and $K=\{1\leq j\leq t\,|\, n_{i_{j}}\geq 2\}$ of $I$. We show that $c_{2}(\rho(\lambda))\in D$ for all $\lambda\in T^{*}$. 
First, assume that $t=1$, i.e., $\lambda=a_{i,1}w_{i,1}+\cdots +a_{i,n_{i}}w_{i, n_{i}}$ for some $i$ and $a_{i,1}, \ldots, a_{i,n_{i}}\in \Z$. If $a_{i,n_{i}}$ is even, then $\lambda\in (\Lambda_{i})_{r}$, thus by (\ref{adjoint}) we have $c_{2}(\rho(\lambda))\in D_{2}\oplus D_{4}$. Otherwise, as $\lambda\in T^{*}$ is equivalent to $e_{i}\in R$, by (\ref{decsimply}) we get $c_{2}(\rho(\lambda))\in D_{1}\oplus D_{2}$. 

Now we assume that $t=2$ and $n_{i_{1}}=n_{i_{2}}=1$, i.e., $\lambda=a_{i}w_{i,1}+a_{j}w_{j,1}$ for some $i, j$ and $a_{i}, a_{j}\in \Z\backslash \{0\}$ with $n_{i}=n_{j}=1$. If both $a_{i}$ and $a_{j}$ are even, then $\lambda\in (\Lambda_{i})_{r}\oplus (\Lambda_{j})_{r}$, so $c_{2}(\rho(\lambda))\in D_{3}\oplus D_{4}$. If $a_{i}$ is even and $a_{j}$ is odd, then as $\lambda\in T^{*}$ if and only if $e_{j}\in R_{1}'$, we get $c_{2}(\rho(\lambda))\in D_{1}\oplus D_{3}\oplus D_{4}$. Similarly, if both $a_{i}$ and $a_{j}$ are odd, then by (\ref{decctwo}) we have $c_{2}(\rho(\lambda))\in D_{3}$.

Finally, assume that either $t\geq 3$ or $t=2$ with $n_{i_{1}}n_{i_{2}}\neq 1$. Then, by the action of the normal subgroups $(\Z/2\Z)^{n_{i}}$ of the Weyl group generated by sign switching, we see that the coefficient at each $e_{i_{j},l}$ in the expansion of $c_{2}(\rho(\lambda))$ is divisible by $4$ and $2$ for $j\in J$ and $j\in K$, respectively, i.e., 
\begin{equation}\label{newtypeBctwo}
	c_{2}(\rho(\lambda))=4(\sum_{j\in J} a_{j}q_{j})+2(\sum_{j\in K}b_{j}q_{j}) 
\end{equation}	
for some $a_{i}, b_{i}\in \Z$. Hence, $c_{2}(\rho(\lambda))\in D$, i.e., $\Dec(G)\subseteq D$.\end{proof}

\subsection{Type $C$}
Let $G=(\prod_{i=1}^{m}\gSp_{2n_{i}})/\gmu$ be a split semisimple group of type $C$, $m, n_{i}\geq 1$, where $\gmu$ is a central subgroup. As $c_{2}(\rho(e_{1}))=-q$, we have $\Dec(\gSp_{2n})=\Z q$. Similarly, as $c_{2}(\rho(2e_{1}))=-4q$ and $c_{2}(\rho(e_{1}+e_{2}))=-2(n-1)q$, we have $\tfrac{4}{\gcd(2, n)}q\in \Dec(\gPGSp_{2n})$. Moreover, since the Weyl group of $\gSp_{2n}$ contains a normal subgroup $(\Z/2\Z)^{n}$ generated by sign switching, we see that $\tfrac{4}{\gcd(2, n)} \,|\,c_{2}(\rho(\lambda))$ for any $\lambda\in \Lambda_{r}$ (c.f. \cite[Part II, \S
14]{GMS}), thus $\Dec(\gPGSp_{2n})=\tfrac{4}{\gcd(2, n)}\Z q$ (see \cite[\S 4b]{Mer163}). Therefore, by (\ref{decsum}) we have
\begin{equation}\label{deccbetween}
\delta_{1}''\Z q_{1}\oplus \cdots \oplus \delta_{m}''\Z q_{m}\subseteq \Dec(G)\subseteq \Z q_{1}\oplus \cdots \oplus\Z q_{m}, \text{ where } \delta_{i}''=\begin{cases}
4 & \text{ if } n_{i} \text{ odd},\\
2 & \text{ if } n_{i} \text{ even}.
\end{cases}
\end{equation}
Similar to the case of type $B$, we determine the subgroup $\Dec(G)$ for type $C$.

\begin{proposition}\label{decGC}
	Let $G=(\prod_{i=1}^{m}\gSp_{2n_{i}})/\gmu$, $m, n_{i}\geq 1$, where $\gmu\simeq (\gmu_{2})^{k}$ is a central subgroup. Let $R$ be the subgroup of $(\gmu_{2}^{m})^{*}=(\Z/2\Z)^{m}$ such that $\gmu^{*}=(\gmu_{2}^{m})^{*}/R$. Let $R_{2}''=\langle e_{i}+e_{j}\in R\,|\, e_{i}, e_{j}\not\in R,\, n_{i}\equiv n_{j}\equiv 1 \mod 2\rangle$ be a subspace of $R$ with $\dim R_{2}''=l_{2}$. Then,
	\begin{equation}\label{propdecC}
	\Dec(G)= (\bigoplus_{e_{i}\in R}\Z q_{i})\oplus (\bigoplus_{n_{i}\equiv 0\!\!\!\!\mod 2,\, e_{i}\not\in R} 2\Z q_{i}) \oplus (\bigoplus_{r=1}^{l_{2}} 2\Z q'_{r})\oplus (\bigoplus_{s=1}^{l_{3}} 4\Z q''_{s}),        
	\end{equation}
	where $l_{3}=|\{i \,|\, n_{i}\equiv 1 \mod 2, e_{i}\not\in R\}|-l_{2}$ and $q_{r}'$ $($resp. $q_{s}'')$ is of the form $q_{i}+q_{j}$  $($resp. $q_{i})$ for some $i, j$ such that $\langle q_{r}', q_{s}''\,|\, 1\leq r\leq l_{2}, 1\leq s\leq l_{3}\rangle=\langle q_{i}\,|\, n_{i}\equiv 1 \mod 2, e_{i}\not\in R\rangle$ over $\Z$.
\end{proposition}
\begin{proof}
Let $T$ be the split maximal torus of $G$. Then, by (\ref{charTC}) we have
\begin{equation}\label{tprimeeq}
T^{*}=\{\sum a_{i,j}e'_{i,j}+\sum a_{i}e_{i,1}\,|\, f_{p}(a_{1},\ldots, a_{m})\equiv 0 \mod 2    \},
\end{equation}
where $e'_{i,j}=e_{i,j}-e_{i,1}$ for all $1\leq i\leq m$ and $2\leq j\leq n_{i}$. First note that we have
\begin{equation}\label{decctwotypetypeC}
-c_{2}(\rho(\chi))=\begin{cases}
a_{i}^{2}q_{i} & \text{ if } \chi\in W(a_{i}e_{i,1}),\\
2(n_{j}a_{i}^{2}q_{i}+n_{i}a_{j}^{2}q_{j}) & \text{ if } \chi\in W(a_{i}e_{i,1}+a_{j}e_{j,1})
\end{cases}
\end{equation}
for any nonzero integers $a_{i}$ and $a_{j}$. We shall denote by $D$ the right hand side of equation (\ref{propdecC}) and write $D=\bigoplus D_{u}$, where $D_{u}$ denotes u-th direct summand of $D$ for $1\leq u\leq 4$. If $e_{i}\in R$, then by (\ref{tprimeeq}) we get $e_{i,1}\in T^{*}$, thus by (\ref{decctwotypetypeC}) $D_{1}\subseteq\Dec(G)$. Similarly, by (\ref{deccbetween}) we have $D_{2}\oplus D_{4}\subseteq\Dec(G)$. Let $e_{i}+e_{j}\in R_{2}''$. Then, by (\ref{tprimeeq}) we have $e_{i,1}+e_{j,1}\in T^{*}$. As both $n_{i}$ and $n_{j}$ are odd, by (\ref{decctwotypetypeC}) we get $2q_{i}+2q_{j}\in \Dec(G)$, i.e., $D_{2}\subseteq \Dec(G)$. Therefore, we get $D\subseteq \Dec(G)$.

Conversely, we shall now show that $c_{2}(\rho(\lambda))\in D$ for all $\lambda\in T^{*}$. Let $\lambda$ be a character written as in $(\ref{lambbbdaexp})$ for some subsets 
\begin{equation}\label{somesubsetjk}
J=\{1\leq j\leq t\,|\, n_{i_{j}}\equiv 1 \mod 2\} \text{ and } K=\{1\leq j\leq t\,|\, n_{i_{j}}\equiv 0 \mod 2\}
\end{equation}
 of $I$. For each $\lambda_{i}=a_{i,1}e_{i,1}+\cdots +a_{i,n_{i}}e_{i, n_{i}}\in \Lambda_{i}$ we shall denote by $|\lambda_{i}|$ the number of nonzero coefficients in $\lambda_{i}$. We first assume that $t=1$, i.e., $\lambda=a_{i,1}w_{i,1}+\cdots +a_{i,n_{i}}e_{i, n_{i}}$ for some $i$ and $a_{i,1}, \ldots, a_{i,n_{i}}\in \Z$. Let $a_{i}=a_{i,1}+\cdots + a_{i, n_{i}}$. By the same argument as in the proof of Proposition \ref{decG}, we have $c_{2}(\rho(\lambda))\in D_{2}\oplus D_{4}$ (resp. $c_{2}(\rho(\lambda))\in D_{1}$) if $a_{i}$ is even (resp. odd). Now we assume that $t=2$ with $|\lambda_{i_{1}}|+|\lambda_{i_{2}}|=2$, i.e., $\lambda=a_{i}e_{i,1}+a_{j}e_{j,1}$ for some $i, j$ and $a_{i}, a_{j}\in \Z\backslash\{0\}$. Then, by the same argument as in the proof of Proposition \ref{decG} we see from (\ref{decctwotypetypeC}) that $c_{2}(\rho(\lambda))\in D$.  

Assume that either $t\geq 3$ or $t=2$ with $|\lambda_{i_{1}}|+|\lambda_{i_{2}}|\geq 3$. Then, as before it follows from the action of the normal subgroups $(\Z/2\Z)^{n_{i}}$ of $W$ that 
\begin{equation*}
c_{2}(\rho(\lambda))=4(\sum_{i\in J} a_{i}q_{i})+2(\sum_{i\in K}b_{i}q_{i}) 
\end{equation*}	
for some $a_{i}, b_{i}\in \Z$. Therefore, we get $c_{2}(\rho(\lambda))\in D$, thus $\Dec(G)\subseteq D$.\end{proof}

\subsection{Type $D$}

Let $G=(\prod_{i=1}^{m}\gSpin_{2n_{i}})/\gmu$ be a split semisimple group of type $D$, where $m\geq 1$, $n_{i}\geq 3$ and $\gmu$ is a central subgroup. Consider the case when $G$ is simple (i.e., $m=1$ and $n_{1}=n$). First of all, as 
\begin{equation}\label{typeDc2}
c_{2}(\rho(\omega_{1}))=-2q,\, c_{2}(\rho(2\omega_{1}))=-8q,\, c_{2}(\rho(\omega_{2}))=\begin{cases} -4(n-1)q & \text{ if } n\geq 4,\\ -q &\text{ if } n=3,\end{cases}
\end{equation}
we have $2\Z q\subseteq \Dec(\gSpin_{2n})$ for $n\geq 4$, $\Dec(\gSpin_{6})=\Z q$, $\tfrac{8}{\gcd(2, n)}\Z q\subseteq \Dec(\gPGO^{+}_{2n})$. On the other hand, as the Weyl group of $\gSpin_{2n}$ contains a normal subgroup $(\Z/2\Z)^{n-1}$ generated by sign switching of even number of coordinates, we see that
$2\,|\,c_{2}(\rho(\lambda))$ for any $\lambda\in \Lambda$ with $n\geq 4$ and $\tfrac{8}{\gcd(2, n)}\,|\,c_{2}(\rho(\lambda'))$ for all $\lambda'\in \Lambda_{r}$ with $n\geq 3$ (c.f. \cite[Part II, \S15]{GMS}), thus $\Dec(\gSpin_{2n})=2\Z q$ for any $n\geq 4$ and $\Dec(\gPGO^{+}_{2n})=\tfrac{8}{\gcd(2, n)}\Z q$ for any $n\geq 3$ (see \cite[\S 4b]{Mer163}). Hence, by (\ref{decsum}) we obtain 
\begin{equation}\label{decDbetween}
\delta_{1}''\Z q_{1}\oplus \cdots \oplus \delta_{m}''\Z q_{m}\subseteq \Dec(G)\subseteq \delta_{1}'\Z q_{1}\oplus \cdots \oplus \delta_{m}'\Z q_{m}, \text{ where } 
\end{equation}
\[\delta_{i}''=\begin{cases}
8 & \text{ if } n_{i} \text{ odd},\\
4 & \text{ if } n_{i} \text{ even},
\end{cases}
\text{ and } \delta_{i}'=\begin{cases}
2 & \text{ if } n_{i}\geq 4,\\
1 & \text{ if } n_{i}=3.
\end{cases}
\]

For the remaining simple groups $\gO^{+}_{2n}$ and $\gHSpin_{2n}$ ($n$ even), we also have $2\Z q\subseteq \Dec(\gO^{+}_{2n})$ and $4\Z q\subseteq \Dec(\gHSpin_{2n})$ by (\ref{typeDc2}). Moreover, if $n=4$, then we have
\begin{equation}\label{ctwonfour}
c_{2}(\rho(\omega_{3}))=c_{2}(\rho(\omega_{4}))=-2q,
\end{equation}
thus  $2\Z q\subseteq \Dec(\gHSpin_{8})$. Then, by the action of the Weyl group as above we obtain $\Dec(\gO^{+}_{2n})=2\Z q$ for all $n\geq 3$,  $\Dec(\gHSpin_{2n})=4\Z q$ for even $n\geq 6$, and $\Dec(\gHSpin_{8})=2\Z q$ (\cite[Theorem 5.1]{BR}). In general, we determine the subgroup $\Dec(G)$ for type $D$.

\begin{proposition}\label{decGD}
Let $G=(\prod_{i=1}^{m}\gSpin_{2n_{i}})/\gmu$, $m\geq 1$, $n_{i}\geq 3$, where $\gmu$ is a central subgroup. Let $R$ be the subgroup of $($\ref{centerD}$)$ such that $\gmu^{*}=Z/R$, $R_{1,i}^{}=R\cap Z_{i}$ for odd $n_{i}$, and $R'_{1,i}=R\cap Z_{i}$ for even $n_{i}$. Set
\begin{align*}
	&I_{1}'=\{i \,\,|\,\, R_{1,i}\neq 0, n_{i}\neq 3\}\cup  \{i \,\,|\,\,R_{1,i}=2Z_{i}, n_{i}=3\}\cup \{i \,\,|\,\, R_{1,i}'\neq 0,\, n_{i}=4\} \,\cup\\ &\{i \,\,|\,\, e_{i,1}+e_{i,2}\in R_{1,i}',\, n_{i}\geq 6\},\,\,  I_{2}'=\{i \,\,|\,\, R_{1,i}'=0\}\cup \{i \,\,|\,\, e_{i,1}+e_{i,2}\not\in R_{1,i}'\neq 0,\, n_{i}\geq 6\}.
\end{align*}
Then, we have
\begin{equation}\label{propdecD}
\Dec(G)=\big(\bigoplus_{R_{1,i}=Z_{i},\, n_{i}=3} \Z q_{i}\big) \oplus \big(\bigoplus_{i\in I_{1}'} 2\Z q_{i}\big) \oplus \big(\bigoplus_{i\in I_{2}'} 4\Z q_{i}\big) \oplus \big(\bigoplus_{r=1}^{l_{2}} 4\Z q_{r}'\big)\oplus \big( \bigoplus_{s=1}^{l_{3}} 8\Z q_{s}'' \big),
\end{equation}
where $l_{2}=\dim_{\Z/2\Z}\langle e_{i}+e_{j}\,\,|\,\, 2e_{i}+2e_{j}\in R, R_{1,i}=R_{1,j}=0\rangle$, $l_{3}=|\{i \,|\, R_{1,i}=0 \}|-l_{2}$,
and $q_{r}'$ $($resp. $q_{s}'')$ is of the form $q_{i}+q_{j}$  $($resp. $q_{i})$ for some $i$, $j$ such that $\langle q_{i}\,|\, R_{1,i}=0\rangle=\langle q_{r}', q_{s}''\,|\, 1\leq r\leq l_{2}, 1\leq s\leq l_{3}\rangle$ over $\Z$.
\end{proposition}
\begin{proof}
Let $\gmu \simeq (\gmu_{2})^{k_{1}}\times (\gmu_{4})^{k_{2}}$ be a central subgroup for some $k_{1}, k_{2}\geq 0$ with $k=k_{1}+k_{2}$. We denote by $D$ the right hand side of equation (\ref{propdecD}) and we write $D=\bigoplus D_{u}$, where $D_{u}$ denotes u-th direct summand of $D$ for $1\leq u\leq 5$. If $e_{i}\in R$ with $n_{i}=3$, then by (\ref{typeDc2}) $D_{1}\subseteq \Dec(G)$. If $2e_{i}\in R$ or $e_{i,1}+e_{i,2}\in R$, then by (\ref{charTD}) we have $w_{i,1}\in T^{*}$, thus by (\ref{typeDc2}) $2q_{i}\in \Dec(G)$. Similarly, if $e_{i1}\in R_{1,i}'$ (resp. $e_{i,2}\in R_{1,i}'$) with $n_{i}=4$, then by (\ref{charTD}) we have $w_{i,3}\in T^{*}$ (resp. $w_{i,4}\in T^{*}$), thus by (\ref{ctwonfour}) $2q_{i}\in \Dec(G)$. Therefore, $D_{2}\subseteq\Dec(G)$. By a simple calculation, we have
\begin{equation}\label{typedtwocomp}
-c_{2}(\rho(\chi))=4(n_{j}a_{i}^{2}q_{i}+n_{i}a_{j}^{2}q_{j}) \text{ if } \chi\in W(a_{i}w_{i,1}+a_{j}w_{j,1})
\end{equation}
for any nonzero integers $a_{i}$ and $a_{j}$. If $2e_{i}+2e_{j}\in R$ for some $i\neq j$, then again by (\ref{charTD}) we obtain $w_{i,1}+w_{j,1}\in T^{*}$. As both $n_{i}$ and $n_{j}$ are odd, by (\ref{typedtwocomp}) $D_{4}\subseteq \Dec(G)$. Finally, it follows by (\ref{decDbetween}) that $D_{3}\oplus D_{5}\subseteq \Dec(G)$, thus $D\subseteq \Dec(G)$.

Now we prove that $c_{2}(\rho(\lambda))\in D$ for all $\lambda\in T^{*}$. Let $\lambda$ be a character written as in $(\ref{lambbbdaexp})$ for some subsets $J$ and $K$ in (\ref{somesubsetjk}). Assume that $t=1$, i.e., $\lambda=a_{i,1}w_{i,1}+\cdots +a_{i,n_{i}}w_{i,n_{i}}$. Applying the same argument as in the proof of Proposition \ref{decGC} we obtain
\[c_{2}(\rho(\lambda))\in \begin{cases} D_{4}\oplus D_{5} & \text{ if } A_{i}=0 \text{ with odd }n_{i},\\ D_{2} & \text{ if } A_{i}\neq 0 \text{ with odd } n_{i}\geq 5 ;\text{ or } A_{i}=2e_{i} \text{ with } n_{i}=3,\\ D_{1} & \text{ if } A_{i}=\pm e_{i} \text{ with } n_{i}=3,\end{cases} \]
where $A_{i}$ denotes the image of $\lambda$ in $Z$ as defined in (\ref{defaii}) and
\[c_{2}(\rho(\lambda))\in \begin{cases} D_{3} & \text{ if } A_{i}=0 \text{ with even }n_{i} ;\text{ or } A_{i}\neq e_{i,1}+e_{i,2} \text{ with even }n_{i}\geq 6,\\ D_{2} & \text{ if } A_{i}\neq 0 \text{ with } n_{i}=4 ;\text{or } A_{i}=e_{i,1}+e_{i,2} \text{ with even }n_{i}\geq 6.\end{cases}\]

We assume that $t=2$ with $\lambda_{i_{1}}$ with $|\lambda_{i_{1}}|+|\lambda_{i_{2}}|=2$. Then, by the same argument as in the proof of Proposition \ref{decG} together with (\ref{typedtwocomp}) we get $c_{2}(\rho(\lambda))\in D$. Finally, Assume that either $t\geq 3$ or $t=2$ with $|\lambda_{i_{1}}|+|\lambda_{i_{2}}|\geq 3$. Then, by the action of the normal subgroups $(\Z/2\Z)^{n_{i}-1}$ of the Weyl group of $G$ we obtain 
\begin{equation*}
	c_{2}(\rho(\lambda))=8(\sum_{i\in J} a_{i}q_{i})+4(\sum_{i\in K}b_{i}q_{i}) 
\end{equation*}	
for some $a_{i}, b_{i}\in \Z$, thus, $c_{2}(\rho(\lambda))\in \Dec(G)$. Hence, $\Dec(G)\subseteq D$.\end{proof}

\section{Degree $3$ invariants for semisimple groups $G$ of types $B$, $C$, $D$}

We now determine the group of reductive indecomposable invariants of split semisimple groups of types $B$, $C$, and $D$ by using the results of Section \ref{QG}, Propositions \ref{decG}, \ref{decGC}, and \ref{decGD}.

\subsection{Type $B$}

\begin{theorem}\label{mainthm}
Let $G=(\prod_{i=1}^{m}\gSpin_{2n_{i}+1})/\gmu$, $m, n_{i}\geq 1$, where $\gmu\simeq (\gmu_{2})^{k}$ is a central subgroup. Let $R$ be the subgroup of $(\gmu_{2}^{m})^{*}=(\Z/2\Z)^{m}$ whose quotient is the character group $\gmu^{*}$. Then,
\[ \Inv^{3}(G)_{\red}=(\Z/2\Z)^{m-k-l_{1}-l_{2}},\,\, \text{where}\]
$l_{1}=\dim \langle e_{i}\in R\,|\, n_{i}\leq 2\rangle$ and $l_{2}=\dim \langle e_{i}+e_{j}\in R\,|,\, e_{i}, e_{j}\not\in R,\, n_{i}=n_{j}=1\rangle$.
\end{theorem}
\begin{proof}
Let $R=\{r=(r_{1},\ldots, r_{m})\in (\Z/2\Z)^{m}\,|\, f_{p}(r)=0, 1\leq p\leq k\}$ be the subgroup of $(\Z/2\Z)^{m}$ whose quotient is the character group $\gmu^{*}$ for some linear polynomials $f_{p}\in \Z/2\Z[t_{1},\ldots, t_{m}]$. Let $\alpha_{i,j}$ denote the simple roots of the $i$th component of the root system of $G$ and let $\theta_{i,j}$ be the square of the length of the coroot of $\alpha_{i,j}$. Then, we have 
\begin{equation*}
	\theta_{i,j}=\begin{cases}
		2 & \text{ if } j=n_{i}\geq 2,\\
		1 & \text{ if } n_{i}\geq 2, 1\leq j\leq n_{i}-1; \text{ or } j=n_{i}=1.
	\end{cases}
\end{equation*}
By \cite[Proposition 7.1]{LM}, an indecomposable invariant of $G$ corresponding to $q=\sum_{i=1}^{m}d_{i}q_{i}\in Q(G)$ is reductive indecomposable if and only if the order $|\bar{w}_{i,j}|$ in $\Lambda/T^{*}$ divides $\theta_{i,j}d_{i}$ for all $i$ and $j$.

Since $|\bar{w}_{i,1}|=1$ with $n_{i}=1$ is equivalent to $e_{i}\in R$ and 
\begin{equation*}
|\bar{w}_{i,j}|\leq\begin{cases}
2 & \text{ if } j=n_{i}\geq 2,\\
1 & \text{ if } n_{i}\geq 2, 1\leq j\leq n_{i}-1,
\end{cases}
\end{equation*}
we see that the equation (\ref{secondeq}) becomes trivial and we may assume that the term $\delta_{i}d_{i} $($=d_{i}$) appears in the equation (\ref{firsteqprime}) is divisible by $2$. Therefore, any reductive indecomposable invariant of $G$ corresponding to $q=\sum_{i=1}^{m}d_{i}q_{i}\in Q(G)$ satisfies
\[f_{p}(\frac{\delta_{1}d_{1}}{2}, \ldots, \frac{\delta_{m}d_{m}}{2})\equiv 0 \mod 2, \text{ where } \delta_{i}=\begin{cases}
		2 & \text{ if } n_{i}\geq 2 \text{ or } e_{i}\in R,\\
		1 & \text{ if } n_{i}=1 \text{ and } e_{i}\not\in R.
		\end{cases}\]
for all $p$, thus we have
\begin{equation}\label{elementinQ}
\Inv^{3}(G)_{\red}=\frac{\{\sum_{i=1}^{m}d_{i}q_{i}\,|\,  f_{p}(\tfrac{\delta_{1}d_{1}}{2}, \ldots, \tfrac{\delta_{m}d_{m}}{2})\equiv 0 \mod 2  \}}{\Dec(G)}.
\end{equation}

Let $R'=R\cap (\bigoplus_{e_{i}\not\in R}(\Z/2\Z) e_{i})$. Then, the group in the numerator of (\ref{elementinQ}) is generated by
\[\{q_{i}\,|\, e_{i}\in R\}\cup \{\sum_{i=1}^{m}\big(\frac{2r_{i}}{\delta_{i}}\big)q_{i}\,| r=(r_{1},\ldots, r_{m})\in R'\}\cup\{\big(\frac{4}{\delta_{i}}\big)q_{i}\,|\, e_{i}\not\in R\}.\]
Hence, the statement for the group of indecomposable reductive invariants follows by Proposition \ref{decG}.\end{proof}

In particular, under the assumption that the ranks of all components of the root system of $G$ are at least $2$ we have the following result.
\begin{corollary}\label{coromain}
Let $G=(\prod_{i=1}^{m}\gSpin_{2n_{i}+1})/\gmu$, $m, n_{i}\geq 1$, where $\gmu\simeq (\gmu_{2})^{k}$ is a central subgroup. Let $R$ be the subgroup of $(\gmu_{2}^{m})^{*}=(\Z/2\Z)^{m}$ whose quotient is the character group $\gmu^{*}$. Assume that $n_{i}\geq 2$ for all $1\leq i\leq m$. Then,
\[ \Inv^{3}(G)_{\ind}=\Inv^{3}(G)_{\red}=(\Z/2\Z)^{m-k-l},\]
where $l=\dim\langle e_{i}\in R\,|\, n_{i}=2\rangle$.
\end{corollary}
\begin{proof}
By Theorem \ref{mainthm}, it suffices to show that $\Inv^{3}(G)_{\ind}\subseteq \Inv^{3}(G)_{\red}$. Since $n_{i}\geq 2$ for all $1\leq i\leq m$, the inclusion follows directly from the proof of Theorem \ref{mainthm}.
\end{proof}
\begin{remark}
One can directly compute $\Inv^{3}(G)_{\ind}$ using Propositions \ref{QGprop} and \ref{decG}.
\end{remark}

We present below another particular case of Theorem \ref{mainthm} (and Theorem \ref{mainthmC}), which follows by the exceptional isomorphism $A_{1}=B_{1}=C_{1}$. This result in turn determine the reductive invariants of semisimple groups of type $A$ (see \cite[Theorem 7.1]{Mer161}).
\begin{corollary}\label{corone}
Let $G=(\prod_{i=1}^{m}\gSL_{2})/\gmu$, $m\geq 1$, where $\gmu\simeq (\gmu_{2})^{k}$ is a central subgroup. Let $R$ be the subgroup of $(\gmu_{2}^{m})^{*}=(\Z/2\Z)^{m}$ whose quotient is the character group $\gmu^{*}$. Then,
\[ \Inv^{3}(G)_{\red}=(\Z/2\Z)^{m-k-l_{1}-l_{2}},\]
where $l_{1}=\dim \langle e_{i}\in R \rangle$ and $l_{2}=\dim \langle e_{i}+e_{j}\in R\,|\, e_{i}, e_{j}\not\in R \rangle$.
\end{corollary}

\subsection{Type $C$}

\begin{theorem}\label{mainthmC}
	Let $G=(\prod_{i=1}^{m}\gSp_{2n_{i}})/\gmu$, $m, n_{i}\geq 1$, where $\gmu\simeq (\gmu_{2})^{k}$ is a central subgroup. Let $R$ be the subgroup of $(\gmu_{2}^{m})^{*}=\bigoplus_{i=1}^{m}(\Z/2\Z) e_{i}$ whose quotient is the character group $\gmu^{*}$ and let $s$ denote the number of ranks $n_{i}$ which are divisible by $4$. Then,
	\[ \Inv^{3}(G)_{\red}=(\Z/2\Z)^{s+l-l_{1}-l_{2}},\,\, \text{where}\]
	$l_{1}=\dim \langle e_{i}\,|\, e_{i}\in R\rangle$, $l_{2}=\dim \langle e_{i}+e_{j}\,|\, e_{i}+e_{j}\in R,\, e_{i}, e_{j}\not\in R,\, n_{i}\equiv n_{j}\equiv 1\mod 2\rangle$, and $l=\dim \big(R\cap (\bigoplus_{4\nmid n_{i}}(\Z/2\Z) e_{i})\big)$. In particular, if $n_{i}\equiv 0 \mod 2$ for all $1\leq i\leq m$, then 
	\[\Inv^{3}(G)_{\ind}=\Inv^{3}(G)_{\red}=(\Z/2\Z)^{s+l-l_{1}}.\]
\end{theorem}
\begin{proof}
We apply arguments similar to the proof of type $B$. Let $\theta_{ij}$ be the square of the length of the coroot corresponding to the simple root of $i$th component of the root system of $G$. Then, we have
\begin{equation*}
\theta_{i,j}=\begin{cases}
1 & \text{ if } j=n_{i}\geq 1,\\
2 & \text{ otherwise.}
\end{cases}
\end{equation*}

Note that $|\bar{w}_{i, n_{i}}|=2$ if and only if $n_{i}$ is odd and the element $e_{i,n_{i}}$ has order $2$ in $\Lambda/T^{*}$. Moreover, by (\ref{charTC}) the latter is equivalent to $e_{i}\not\in R$. Hence, by \cite[Proposition 7.1]{LM} an indecomposable invariant of $G$  corresponding to $q=\sum_{i=1}d_{i}q_{i}\in Q(G)$ is reductive indecomposable if and only if $2|d_{i}$ for all odd $n_{i}$ such that $e_{i}\not\in R$. Therefore, any reductive indecomposable invariant of $G$ corresponding to $q=\sum_{i=1}d_{i}q_{i}\in Q(G)$ obviously satisfies the first equation of (\ref{firsteqtypeC}) and the second equation of (\ref{firsteqtypeC}) divided by $2$, i.e.,
\[f_{p}(\tfrac{\delta_{1}n_{1}d_{1}}{2},\ldots, \tfrac{\delta_{m}n_{m}d_{m}}{2})\equiv 0 \mod 2, \text{ where } \delta_{i}=\begin{cases}
1 & \text{ if } e_{i}\not\in R,\\
\tfrac{2}{\,\,n_{i}} & \text{ if } e_{i}\in R
\end{cases} \]
for all $1\leq p\leq k$, thus, 
\begin{equation}\label{redtypeC}
\Inv^{3}(G)_{\red}=\{\sum_{i=1}^{m}d_{i}q_{i}\,|\,  f_{p}(\tfrac{\delta_{1}n_{1}d_{1}}{2},\ldots, \tfrac{\delta_{m}n_{m}d_{m}}{2})\equiv 0 \mod 2  \}/\Dec(G).
\end{equation}

Let $R'=R\cap (\bigoplus_{4\nmid n_{i}, e_{i}\not\in R}(\Z/2\Z) e_{i})$, where $R$ is the subgroup of $\bigoplus_{i=1}^{m}(\Z/2\Z) e_{i}$ as in (\ref{relationB}). Then, we easily see that the group in the numerator of (\ref{redtypeC}) is generated by
\[\{q_{i}\,|\,e_{i}\in R  \text{ or } 4|n_{i}  \}\cup \{\sum_{i=1}^{m}\epsilon_{i}r_{i}q_{i}\,| r=(r_{i})\in R'\}\cup \{2\epsilon_{i}q_{i}\,|\, e_{i}\not\in R\},\, \epsilon_{i}=\begin{cases}
1 & \!\! \text{ if } 2| n_{i} ,\\
2 & \!\! \text{ if } 2\nmid n_{i}.
\end{cases}\]
Hence, the statement immediately follows from Proposition \ref{decGC}. If $n_{i}$ is even for all $i$, then the same argument together with Proposition \ref{QGpropC} shows the result for the group of indecomposable invariants. \end{proof}

\subsection{Type $D$}

\begin{theorem}\label{mainthmD}
Let $G=(\prod_{i=1}^{m}\gSpin_{2n_{i}})/\gmu$, $m\geq 1$, $n_{i}\geq 3$, where $\gmu$ is a central subgroup. Let $R$ be the subgroup of the character group $Z$ defined in $($\ref{centerD}$)$ such that $\gmu^{*}=Z/R$, $R_{1,i}^{}=R\cap Z_{i}$ for odd $n_{i}$, $R'_{1,i}=R\cap Z_{i}$ for even $n_{i}$, and let
\[\bar{R}=\{(\bar{r}_{1},\ldots, \bar{r}_{m})\in \bigoplus_{i=1}^{m}(\Z/2\Z)\bar{e}_{i}\,|\, \sum_{i=1}^{m}r_{i}\in R\}, r_{i}:=\begin{cases} 2\bar{r}_{i}e_{i} & \text{ if } n_{i} \text{ odd,}\\ \bar{r}_{i}e_{i,1}+\bar{r}_{i}e_{i,2} & \text{ if } n_{i} \text{ even},\end{cases}\] \begin{equation*}
 \text{ where }	Z:=\bigoplus_{i=1}^{m}Z_{i} \text{ with } Z_{i}=\begin{cases} 
		(\Z/4\Z)e_{i} & \text{ if } n_{i} \text{ odd},\\
		(\Z/2\Z)e_{i,1}\bigoplus (\Z/2\Z)e_{i,2} & \text{ if } n_{i} \text{ even}.
	\end{cases}
\end{equation*}
denote the character group of the center of $\prod_{i=1}^{m}\gSpin_{2n_{i}}$. Set
\begin{align*}
&R'=\bar{R}\cap \big(\bigoplus_{4\nmid n_{i}, R_{1,i}', R_{1,i}\neq Z_{i}}(\Z/2\Z)\bar{e}_{i}\big) \text{ with } l=\dim R', \, I_{1}=\{i\,|\, Z_{i}=R_{1,i} \text{ or } R_{1,i}', n_{i}\neq 3\},\\
&I_{2}=\{i \,|\,R_{1,i}'=0,  4|n_{i}\}\cup \{i\,|\, R_{1,i}'\!=\!(\Z/2\Z)e_{i,1} \text{ or } (\Z/2\Z)e_{i,2}, n_{i}\geq 6, 4|n_{i}\} \text{ with } s_{i}=|I_{i}|.\!\!
\end{align*}
Then, we have
\[ \Inv^{3}(G)_{\red}=(\Z/2\Z)^{s_{1}+s_{2}+l-l_{1}-l_{2}},\,\, \text{where}\]
$l_{1}=|\{i \,|\, 4\nmid n_{i}, R_{1,i}=2Z_{i} \text{ or } R_{1i}'=(\Z/2\Z)(e_{i,1}+e_{i,2})\}|$, and $l_{2}=\dim\langle \bar{e}_{i}+\bar{e}_{j}| 2e_{i}+2e_{j}\in R, R_{1,i}=R_{1,j}=0\rangle$.
\end{theorem}
\begin{proof}
Let $Z$ denote the character group of the center of $\prod_{i=1}^{m}\gSpin_{2n_{i}}$ as in (\ref{centerD}). Let $\gmu$ be a central subgroup such that $\gmu\simeq (\gmu_{2})^{k_{1}}\times (\gmu_{4})^{k_{2}}$ for some $k_{1}, k_{2}\geq 0$ and let $R=\{ r\in Z\,|\, f_{p}(r)=0, 1\leq p\leq k\}$
be the subgroup of $Z$ such that $\gmu^{*}\simeq Z/R$ for some linear polynomials $f_{p}\in \Z/4\Z[T_{1},\ldots, T_{m}]$ with $k=k_{1}+k_{2}$. We shall use the description of $Q(G)$ in Section \ref{subD}. 

Let $\theta_{i,j}$ denote the square of the length of the $j$th coroot of the $i$th component of the root system of $G$. Then, $\theta_{i,j}=1$ for all $1\leq i\leq m$ and $1\leq j\leq n_{i}$. Note that the order of the fundamental weight $w_{i,j}$ in $\Lambda/T^{*}$ is trivial for all $j$ if and only if 
\[Z_{i}=\begin{cases} R_{1,i} & \text{ if } n_{i} \text{ odd,}\\ R_{1,i}' & \text{ if } n_{i} \text{ even.}\end{cases}  \]
Moreover, if $c_{i}(p)=\pm 1$ for some $1\leq p\leq k$, where $c_{i}(p)$ denotes the coefficient of $t_{i}$ in $f_{p}$, then $R_{1i}=0$, thus by (\ref{charTD}) $2w_{i, n_{i}}\not\in T^{*}$, i.e., $|\bar{w}_{i,n_{i}}|=4$. Hence, by \cite[Proposition 7.1]{LM} any reductive indecomposable invariant of $G$ corresponding to $q=\sum_{i=1}d_{i}q_{i}\in Q(G)$ satisfies (\ref{firsteqtypeDprime}) and (\ref{secondconditionre}). Therefore, it follows by (\ref{equifirstD}) that
\begin{equation}\label{redtypeD}
\Inv^{3}(G)_{\red}=\frac{\{\sum_{i=1}^{m}d_{i}q_{i}\,|\,  \bar{f}_{p}(\epsilon_{1} d_{1},\cdots, \epsilon_{m}d_{m})\equiv 0 \mod 2  \}}{\Dec(G)}
\end{equation}
where, $\bar{f}_{p}\in \Z/2\Z[t_{1},\ldots, t_{m}]$ denotes the image of $f_{p}$ under the following map
\[\Z/4\Z[T]\to \Z/4\Z[t_{1},\ldots, t_{m}]\to \Z/2\Z[t_{1},\ldots, t_{m}] \text{ given by } 2t_{i1}, 2t_{i2}\mapsto t_{i}, t_{i}\mapsto t_{i}\]
\[\text{ and } \epsilon_{i}=\begin{cases} 1 & \text{ if } Z_{i}=R_{1,i} \text{ or }R_{1,i}',\\ \tfrac{1}{2} & \text{ if } c_{i}(p)=2 \text{ or }  c_{i1}(p)+c_{i2}(p)=4,\\ \tfrac{n_{i}}{4} & \text{ otherwise.}\end{cases}\]

Let $\bar{R}=\{\bar{r}=(\bar{r}_{1},\ldots, \bar{r}_{m})\in \bigoplus_{i=1}^{m}(\Z/2\Z)\bar{e}_{i}\,|\, \bar{f}_{p}(\bar{r})\equiv 0 \mod 2\}$, equivalently
\[\bar{R}=\{(\bar{r}_{1},\ldots, \bar{r}_{m})\in (\Z/2\Z)^{m}\,|\, \sum_{i=1}^{m}r_{i}\in R\}, \text{ where } r_{i}:=\begin{cases} 2\bar{r}_{i}e_{i} & \text{ if } n_{i} \text{ odd,}\\ \bar{r}_{i}e_{i,1}+\bar{r}_{i}e_{i,2} & \text{ if } n_{i} \text{ even}\end{cases} \]
and let $R'=\bar{R}\cap \big(\bigoplus_{4\nmid n_{i}, R_{1,i}', R_{1,i}\neq Z_{i}}(\Z/2\Z)\bar{e}_{i}\big)$. Observe that $\bar{f}_{p}(\bar{e}_{i})\equiv 0 \mod 2$ for all $p$ with $n_{i}$ odd if and only if either $c_{i}(p)=0$ or $2$ for all $p$ (i.e., $f_{p}(e_{i})\equiv 0$ or $f_{p}(2e_{i})\equiv 0 \mod 4$, respectively) and this, in turn, is equivalent to $R_{1,i}=Z_{i}$ or $2Z_{i}$. Similarly, $\bar{f}_{p}(\bar{e}_{i})\equiv 0 \mod 2$ for all $p$ with $n_{i}$ even if and only if either $c_{i1}(p)=c_{i2}(p)=0$ or $c_{i1}(p)=c_{i2}(p)=2$ for all $p$ (i.e., $f_{p}(e_{i1})\equiv f_{p}(e_{i2})\equiv 0$ or $f_{p}(e_{i1}+e_{i2})\equiv 0 \mod 4$, respectively) and this, in turn, is equivalent to $R_{1,i}'=Z_{i}$ or $(\Z/2\Z)(e_{i1}+e_{i2})$. Therefore, we see that the group in the numerator of (\ref{redtypeD}) is generated by
\begin{align*}
&\{2q_{i}\,|\, R_{1,i}=2Z_{i}\text{ or } R_{1,i}'=(\Z/2\Z)(e_{i1}+e_{i2})\text{ or } e_{i,1}+e_{i,2}\not\in R'_{1,i}, 4|n_{i}  \}\cup \{8q_{i}\,|\, R_{1,i}=0\}\\
&\cup \{q_{i}\,|\, Z_{i}=R_{1,i} \text{ or } R'_{1,i}\}\cup \{4q_{i}\,|\, e_{i,1}+e_{i,2}\not\in R_{1,i}', 4\nmid n_{i}  \}\cup  \{\sum_{i=1}^{m}\epsilon_{i}'r_{i}'q_{i}\}
\end{align*}
for all $r'=(r_{1}',\ldots, r_{m}')\in R'\backslash \{\bar{e}_{i}\,|\, 1\leq i\leq m\}$, where $\epsilon_{i}'= 2$ (resp. $4$)  if $n_{i}$ is even (resp. odd). Therefore, the statement immediately follows by Proposition \ref{decGD}.\end{proof}

\section{Unramified invariants for semisimple groups $G$ of types $B$, $C$, $D$}

In this section, we first describe torsors for the corresponding reductive groups in Lemmas \ref{honegredprop}, \ref{honegredpropC}, and \ref{honegredpropD}. Then, using this together with Theorems \ref{mainthm}, \ref{mainthmC}, and \ref{mainthmD}, we present a complete description of the corresponding cohomological invariants in Propositions \ref{formcor}, \ref{formcorC}, and \ref{formcorD}. Finally, using such descriptions, we show that there are no nontrivial unramified degree $3$ invariants for semisimple groups of types $B$, $C$, and $D$ (see Theorems \ref{secthm}, \ref{secthmC}, \ref{secthmD}). In this section, we assume that the base field $F$ is of characteristic $0$. We shall write $\langle a_{1},\ldots, a_{n}\rangle=\langle a_{1}\rangle \perp \cdots \perp \langle a_{n}\rangle$ for the diagonal quadratic form $a_{1}x_{1}^{2}+\cdots +a_{n}x_{n}^{2}$ and write $\langle\langle a_{1},\ldots, a_{n} \rangle\rangle=\langle 1, -a_{1}\rangle\tens \cdots \tens\langle 1, -a_{n}\rangle $ for the $n$-fold Pfister form.

\subsection{Type $B$}

\begin{lemma}\label{honegredprop}
Let $G=(\prod_{i=1}^{m}\gSpin_{2n_{i}+1})/\gmu$, $m, n_{i}\geq 1$, where $\gmu$ is a central subgroup. Let $R$ be the subgroup of $(\gmu_{2}^{m})^{*}=(\Z/2\Z)^{m}$ whose quotient is the character group $\gmu^{*}$. Set $G_{\red}=(\prod_{i=1}^{m}\gGamma_{2n_{i}+1})/\gmu$, where $\gGamma_{2n_{i}+1}$ is the split even Clifford group. Then, for any field extension $K/F$ the first Galois cohomology set $H^{1}(K, G_{\red})$ is bijective to the set of $m$-tuples of quadratic forms $(\phi_{1},\ldots, \phi_{m})$ with $\dim \phi_{i}=2n_{i}+1$, $\disc\phi_{i}=1$ such that for all $r=e_{i_{1}}+\cdots +e_{i_{s}}\in R$, $i_{1}<\cdots<i_{s}$,
\begin{equation}\label{conditionI3}
I^{3}(K)\ni\begin{cases}
\perp_{p=1}^{s} (-1)^{p}\phi_{i_{p}} & \text{ if } s \text{ is even},\\
(\perp_{p=1}^{s} (-1)^{p}\phi_{i_{p}})\perp \langle 1\rangle  & \text{ otherwise,}
\end{cases}
\end{equation}
where $\disc\phi_{i}$ denotes the discriminant of $\phi_{i}$ and $I^{3}(K)$ denotes the cubic power of the fundamental ideal $I(K)$ in the Witt ring of $K$.
\end{lemma}
\begin{proof}
Let $G_{\red}=(\prod_{i=1}^{m}\gGamma_{2n_{i}+1})/\gmu$, where $\gGamma_{2n_{i}+1}$ denotes the split even Clifford group. Consider the natural exact sequence
\[1\to (\gm)^{m}/\gmu \to G_{\red}\to \prod_{i=1}^{m}\gO^{+}_{2n_{i}+1}\to 1.\]
Then, by Hilbert Theorem $90$ and \cite[Proposition 42]{Serre}, this sequence yields a bijection between the set $H^{1}(F,G_{\red})$ and the kernel of the connecting map which factors as
\[H^{1}(F, \prod_{i=1}^{m}\gO^{+}_{2n_{i}+1})\stackrel{}\to H^{2}(F, (\gmu_{2})^{m})=\Br_{2}(F)^m\stackrel{\tau}\to H^{2}(F,(\gmu_{2})^{m}/\gmu),\] 
where the first map sends an $m$-tuple of quadratic forms $(\phi_{1},\ldots, \phi_{m})$ with $\dim \phi_{i}=2n_{i}+1$, $\disc(\phi_{i})=1$ to the $m$-tuple $(C_{0}(\phi_{1}), \ldots, C_{0}(\phi_{m}))$ of even Clifford algebras $C_{0}(\phi_{i})$ associated to $\phi_{i}$ and the map $\tau$ is induced by the natural surjection $(\gmu_{2})^{m}\to (\gmu_{2})^{m}/\gmu$. Since $(C_{0}(\phi_{1}), \ldots, C_{0}(\phi_{m}))\in \Ker(\tau)$ if and only if it is contained in the kernel of the composition
\begin{equation}\label{Bdescrip}
H^{2}(F, (\gmu_{2})^{m})\stackrel{\tau}\to H^{2}(F,(\gmu_{2})^{m}/\gmu)\stackrel{r_{*}}\to H^{2}(F, \gm)
\end{equation}
for all $r\in R=((\gmu_{2})^{m}/\gmu)^{*}$, we have
\begin{equation}\label{HoneGred}
H^{1}(F, G_{\red})\simeq \{(\phi_{1},\ldots, \phi_{m})\,|\,\dim \phi_{i}=2n_{i}+1, \disc\phi_{i}=1, \sum_{i=1}^{m}r_{i}C_{0}(\phi_{i})=0         \}   
\end{equation}
for all $r=(r_{i})\in R$.   

Write an element $r\in R$ as $r=e_{i_{1}}+\cdots +e_{i_{s}}$ for some $i_{1}<\cdots <i_{s}$, so that the condition $\sum_{i=1}^{m}r_{i}C_{0}(\phi_{i})=0$ in (\ref{HoneGred}) is equal to $\sum_{p=1}^{s} C_{0}(\phi_{i_{p}})=0$. Assume that $s$ is even. Since $\disc(-\phi_{i_{p}}\perp\phi_{i_{p+1}})=1$ for any $1\leq p\leq s/2$,
\[C_{0}(\psi)=C_{0}(-\psi), \text{ and } C_{0}(\phi)+C_{0}(\phi')=C(\phi\perp \phi')\]
for any quadratic form $\psi$ and any odd-dimensional quadratic forms $\phi$ and $\phi'$, where $C(\phi\perp \phi')$ is the corresponding Clifford algebra, we have
\[0=\sum_{p=1}^{s} C_{0}(\phi_{i_{p}})=C(-\phi_{i_{1}}\perp\phi_{i_{2}}\perp\cdots \perp-\phi_{i_{s-1}}\perp\phi_{i_{s}}),\]
which is equivalent to $(-\phi_{i_{1}}\perp\phi_{i_{2}})\perp\cdots \perp(-\phi_{i_{s-1}}\perp\phi_{i_{s}})\in I^{3}(F)$ by \cite[Theorem 14.3]{EKM}. Now we assume that $s$ is odd. Since $C_{0}(\phi\perp \langle 1 \rangle)=C_{0}(\phi)$ for any odd-dimensional quadratic form $\phi$ and $\disc(-\phi_{i_{s}}\perp \langle 1\rangle)=1$, the same argument shows that
$(-\phi_{i_{1}}\perp\phi_{i_{2}})\perp\cdots \perp(-\phi_{i_{s-2}}\perp\phi_{i_{s-1}})\perp (-\phi_{i_{s}}\perp\langle 1 \rangle)\in I^{3}(F)$.\end{proof}
\begin{remark}\label{remarkforsimple}
If we assume that $-1\in (F^{\times})^{2}$, then the condition (\ref{conditionI3}) in Lemma \ref{honegredprop} can be simplified without sign changes as follows:
\begin{equation*}
H^{1}(K, G_{\red})\simeq \{\phi:=(\phi_{1},\ldots, \phi_{m})\,|\,\dim \phi_{i}=2n_{i}+1, \disc\phi_{i}=1, \phi[r]\in I^{3}(K)         \}   
\end{equation*}
for all $r=(r_{i})\in R$, where 
\[\phi[r]:=\begin{cases}
	\perp_{i=1}^{m} r_{i}\phi_{i} & \text{ if } \sum_{i=1}^{m}r_{i}\equiv 0 \mod 2,\\
	(\perp_{i=1}^{m} r_{i}\phi_{i})\perp \langle 1\rangle  & \text{ otherwise.}
	\end{cases}
\]
\end{remark}

\begin{proposition}\label{formcor}
Let $G=(\prod_{i=1}^{m}\gSpin_{2n_{i}+1})/\gmu$ defined over an algebraically closed field $F$, where $m, n_{i}\geq 1$, $\gmu$ is a central subgroup. Let $R$ be the subgroup of $(\gmu_{2}^{m})^{*}$ whose quotient is the character group $\gmu^{*}$. Set $G_{\red}=(\prod_{i=1}^{m}\gGamma_{2n_{i}+1})/\gmu$, where $\gGamma_{2n_{i}+1}$ is the split even Clifford group. Then, every normalized invariant in $\Inv^{3}(G_{\red})$ is of the form $\boldsymbol{\mathrm{e}}_{3}(\phi[r])$ for some $r\in R$, where $\phi[r]$ is the quadratic form defined in Remark \ref{remarkforsimple} and $\boldsymbol{\mathrm{e}}_{3}: I^{3}(K)\to H^{3}(K)$ denotes the Arason invariant over a field extension $K/F$.  Moreover, we have
\begin{equation}\label{formcoreqa}
\Inv^{3}(G_{\red})_{\norm}\simeq \frac{R}{\langle e_{i},\, e_{j}+e_{k}\in R\,|\, e_{j}, e_{k}\not\in R,\, n_{i}\leq 2,\, n_{j}=n_{k}=1\rangle}.
\end{equation}
\end{proposition}
\begin{proof}
Observe that $\Inv^{3}(G_{\red})_{\norm}=\Inv^{3}(G_{\red})_{\ind}$ as $F$ is algebraically closed. Since $\phi[r]\in I^{3}(K)$ for any $r\in R$, the Arason invariant gives a normalized invariant of $G_{\red}$ of order dividing $2$ that sends an $m$-tuple $\phi\in H^{1}(K, G_{\red})$ to $\boldsymbol{\mathrm{e}}_{3}(\phi[r])\in H^{3}(K)$.

Let $r\in R_{1}'+R_{2}'$, where $R_{1}'$ and $R_{2}'$ denote the subgroups of $R$ defined in Proposition \ref{decG}. Then, as every $4$ and $6$-dimensional quadratic forms in $I^{3}(K)$ are hyperbolic, the invariant $\boldsymbol{\mathrm{e}}_{3}(\phi[r])$ vanishes.

Now we show that the invariant $\boldsymbol{\mathrm{e}}_{3}(\phi[r])$ is nontrivial for any $r\in R\backslash (R_{1}'+R_{2}')$. Let $G'_{\red}=(\gGamma_{3})^{m}/\gmu$. If $R$ is a subgroup such that every element $r$ in $R$ has at least $3$ nonzero components, then by \cite[Lemma 4.3]{Mer17} and the exceptional isomorphism $A_{1}=B_{1}$, any invariant of $G'_{\red}$ is nontrivial. Hence, it follows from the map $$\Inv^{3}(G_{\red})\to \Inv^{3}(G'_{\red})$$ induced by the standard embedding $\gGamma_{3}\to \gGamma_{2n_{i}+1}$ that every invariant $\boldsymbol{\mathrm{e}}_{3}(\phi[r])$ is nontrivial. Otherwise, by the proof of Lemma \ref{lemmaunrami} each invariant $\boldsymbol{\mathrm{e}}_{3}(\phi[r])$ is nontrivial, thus the statements follow from Theorem \ref{mainthm}.\end{proof}

Recall from Section \ref{QG} the following subgroups of $R$.
\begin{equation*}
R_{1}=\langle e_{i}\in R\rangle \text{ and } R_{2}=\langle e_{i}+e_{j}\in R\,|\,  e_{i}, e_{j}\not\in R_{1}\rangle.
\end{equation*}
We shall need the following key lemma.

\begin{lemma}\label{lemmaunrami}
Let $G=(\prod_{i=1}^{m}\gSpin_{2n_{i}+1})/\gmu$ defined over an algebraically closed field $F$, where $m, n_{i}\geq 1$, $\gmu$ is a central subgroup. Set $G_{\red}=(\prod_{i=1}^{m}\gGamma_{2n_{i}+1})/\gmu$. Then, every normalized invariant in $\Inv^{3}(G_{\red})$ is ramified if either $n_{i}\geq 3$ for some $i$ with $e_{i}\in R_{1}$ or $n_{j}+n_{k}\geq 3$ for some $j$ and $k$ such that $e_{j}+e_{k}\in R_{2}$.
\end{lemma}
\begin{proof}
Let $R_{3}$ be a complementary subspace of $R_{1}+R_{2}$ in $R$. Then, by Proposition \ref{formcor} any normalized invariant $\alpha$ in $\Inv^{3}(G_{\red})$ can be written as 
\begin{equation*}
\alpha(\phi)=\boldsymbol{\mathrm{e}}_{3}(\phi[r_{1}])+\boldsymbol{\mathrm{e}}_{3}(\phi[r_{2}])+\boldsymbol{\mathrm{e}}_{3}(\phi[r_{3}])   
\end{equation*}
for some $r_{i}\in R_{i}$, $1\leq i\leq 3$, where $\phi=(\phi_{1},\ldots, \phi_{m})$ denotes a $G_{\red}$-torsor.

Suppose that $r_{1}\in R_{1}$ is nonzero. Then, we may assume that $r_{1}=e_{1}$ with $n_{1}\geq 3$. Choose a division quaternion algebra $(x, y)$ over a field extension $K/F$. Find $\phi[e_{1}]=\phi_{1}$ such that $\phi_{1}\perp \langle 1\rangle=\langle\langle x, y, z \rangle\rangle\perp h$ over the field of formal Laurent series $K((z))$ and set $\phi_{i}=h\perp \langle 1\rangle$ for all $2\leq i\leq m$, where $h$ denotes a hyperbolic form. Then, we have
$\partial_{z}(\alpha(\phi))=(x, y)\neq 0$, where $\partial_{z}$ denotes the residue map, thus $\alpha(\phi)$ ramifies.

Now we may assume that $\alpha(\phi)=\boldsymbol{\mathrm{e}}_{3}(\phi[r_{2}])+\boldsymbol{\mathrm{e}}_{3}(\phi[r_{3}])$ with $r_{2}\neq 0$. To show that $\alpha(\phi)$ ramifies, we shall choose bases of $R_{2}$ and a complementary subspace of $R_{2}$. For simplicity, we will write $e(i_{1},\ldots, i_{k})$ for $e_{i_{1}}+\cdots +e_{i_{k}}$. We first select $e(i_{p}, i_{p,q})\in R_{2}$, where $i_{p}$, $i_{p,q}$ ($1\leq p\leq k$, $1\leq q\leq m_{p}$) are all distinct integers for some $m_{1},\ldots, m_{k}$, so that $B_{2}:=\{e(i_{p}, i_{p,q})\}$ is a basis of $R_{2}$. In particular, if $n_{i_{p,q}}=1$ for some $p$ and $q$, say $n_{i_{1,1}}=1$, then we replace the subset $\{e(i_{1}, i_{1,q})\,|\, 1\leq q\leq m_{1}\}$ of $B_{2}$ by $\{e(i_{1,1}, i_{1}), e(i_{1,1}, i_{1q})\,|\, 2\leq q\leq m_{1} \}$ so that we may assume that $n_{i_{p}}=1$. We set
\[I_{2}=\{i_{p}\,|\, 1\leq p\leq k\}\text{ and } I_{2}'=\{i_{p}, i_{p,q}\,|\, 1\leq p\leq k, 1\leq q\leq m_{p}\}.            \]

We select a basis $B_{3}$ of a complementary subspace of $R_{2}$. First, we find any basis $D_{3}$ of $R_{3}$. Then, we modify each element $d$ of $D_{3}$ by adding $e(i_{p}, i_{p,q})$ to it whenever either $e(i_{p,q})$ or $e(i_{p}, i_{p,q})$ appears in $d$. Hence, we obtain a basis $C_{3}:=\{e(k_{1},\ldots, k_{l})\}$ of a complementary subspace of $R_2$ such that the intersection
\[\big(\bigcup\{k_{1},\ldots, k_{l}\,|\, e(k_{1},\ldots, k_{l})\in C_{3}   \}\big)\cap I_{2}',\]
where the union is over all elements of $C_{3}$, is a subset of $I_{2}$. We denote by $J_{2}$ the intersection. We can divide all elements of the basis $C_{3}$ into two types: either $e(i_{p})$ for some $i_{p}\in J_{2}$ appears in $e(k_{1},\ldots, k_{l})\in C_{3}$ (the first type) or not (the second type).

We first select basis elements from the first type elements as follows. We choose any element $b(i_{1})$ in $C_{3}$ of the first type such that $e(i_{1})$ appears in the element (if there is no element of the first type, we skip the selection of elements of the first type). We write $b(i_{1}):=e(i_{1})+b'(i_{1})$, where $e(i_{1})$ does not appear in $b'(i_{1})$. We modify every element of the first type by adding $b(i_{1})$ to the element whenever $e(i_{1})$ appears in the element. For simplicity, we shall use the same notation $C_{3}$ for the modified basis of $C_{3}$. Then, $e(i_{1})$ appears only in $b(i_{1})$ among the elements of $C_{3}$. Now we choose another element $b(i_{2})$ of the first type in which $e(i_{2})$ appears for some $i_{2}\in J_{2}$. We write $b(i_{2}):=e(i_{2})+b'(i_{2})$, where $e(i_{2})$ does not appear in $b'(i_{2})$. As $e(i_{1})$ appears only in $b(i_{1})$, both $e(i_{1})$ and $e(i_{2})$ do not appear in $b'(i_{2})$. Again, we modify every element of the first type by adding $b(i_{2})$ to the element whenever $e(i_{2})$ appears in the element. In particular, both $e(i_{1})$ and $e(i_{2})$ do not appear in the modified $b'(i_{1})$. We do the same procedure successively for all elements of the first type so that we have chosen basis elements $b(i_{p}):=e(i_{p})+b'(i_{p})$ for all $i_{p}$ in some subset $J_{2}'\subseteq J_{2}$ such that all the terms $e(i_{p})$ do not appear in $b'(i_{p})$.

Similarly, we select basis elements from the second type elements. We choose any element $b(j_{1})$ of the second type with $j_{1}\not\in J_{2}$, so that we write $b(j_{1}):=e(j_{1})+b'(j_{1})$, where $e(j_{1})$ does not appear in $b'(j_{1})$. We modify every element of $C_{3}$ (i.e., $b(i_{p})$ and elements of the second type) by adding $b(j_{1})$ to the element whenever $e(j_{1})$ appears in the element. Then, in particular, all the terms $e(i_{p})$ and $e(j_{1})$ do not appear in the modified $b'(i_{p})$. Now we choose another element $b(j_{2})$ of the second type for some $j_{2}\not\in J_{2}$, so that we have $b(j_{2}):=e(j_{2})+b'(j_{2})$, where both $e(j_{1})$ and $e(j_{2})$ do not appear in $b'(j_{2})$. Again we modify every element of $C_{3}$ by adding $b(j_{2})$ to the element whenever $e(j_{2})$ appears in the element. Then, both $e(j_{1})$ and $e(j_{2})$ do not appear in modified $b'(j_{2})$ and all the terms $e(i_{p})$, $e(j_{1})$, and $e(j_{2})$ do not appear in the modified $b'(i_{p})$. Applying the same procedure to all elements of the second type, we obtain the following basis $B_{3}$ of a complementary subspace of $R_{2}$:
\[b(i_{p}):=e(i_{p})+b'(i_{p}),\, b(j_{1}):=e(j_{1})+b'(j_{1}),\, \cdots    ,b(j_{s}):=e(j_{s})+b'(j_{s})                    \]
for all $i_{p}\in J_{2}'$ and some distinct $j_{1},\ldots, j_{s}\not\in J_{2}$ such that all the terms $e(i_{p})$ and $e(j_{r})$ do not appear in $b'(i_{p}), b'(j_{r})$ for all $1\leq r\leq s$, thus $$B_{3}=\{b(i_{p}), b(j_{r}) \,|\, i_{p}\in J_{2}', 1\leq r\leq s\}.$$

Using the basis $B_{2}\cup B_{3}$, we rewrite the invariant $\alpha(\phi)=\boldsymbol{\mathrm{e}}_{3}(\phi[r_{2}])+\boldsymbol{\mathrm{e}}_{3}(\phi[r_{3}])$ as
\begin{equation}\label{alphaphire}
\alpha(\phi)=\sum_{b\in B_{2}'}\boldsymbol{\mathrm{e}}_{3}(\phi[b])+\sum_{b\in B_{3}'}\boldsymbol{\mathrm{e}}_{3}(\phi[b])
\end{equation}
for some subsets $\emptyset\neq B_{2}'\subseteq B_{2}$ and $B_{3}'\subseteq B_{3}$. Now we show that the invariant $\alpha(\phi)$ in (\ref{alphaphire}) ramifies. It is convenient to split the proof into two cases.

\textbf{Case 1}: $\exists$ $e(i_{p}, i_{p,q})\in B_{2}'$ with $n_{i_{p}}+n_{i_{p,q}}\geq 3$ such that $i_{p}\not\in J_{2}'$. Let $e(i_{u}, i_{u,v})\in B_{2}'$ be such an element for some $1\leq u\leq k$ and $1\leq v\leq m_{u}$ and let $I=\{1,\ldots, m\}$. We take a division quaternion algebra $(x, y)$ over a field extension $K/F$.  Then, choose $\phi_{i}$ for all $i\in I$ such that 
\begin{equation}\label{eueuv}
\phi[e(i_{u})]=\phi[e(i_{u,q})]=\langle x, y, xy\rangle\perp h,\,\, \phi[e(i_{u,v})]=\langle 1, z, xz, yz, xyz\rangle\perp h  
\end{equation}
for all $1\leq q\neq v\leq m_{u}$,
\begin{equation*}
\phi[e(i_{p})]=\phi[e(i_{p,q})]=\begin{cases}\langle x, y, xy\rangle\perp h & \text{ if } e(i_{u}) \text{ appears in } b(i_{p}),\\ \langle 1\rangle\perp h & \text{ otherwise,} \end{cases} 
\end{equation*}
for all $i_{p}\in J_{2}'$ and all $q$ with $e(i_{p},i_{p,q})\in B_{2}$, and $\phi_{i}=\langle 1\rangle\perp h$ for the remaining $i\in I$ over $K((z))$, where $h$ denotes a hyperbolic form depending on the dimension of each $\phi_{i}$. Then, we have 
\begin{equation}\label{uuvuuq}
\phi[e(i_{u},i_{u,v})]=\langle\langle x, y, z\rangle\rangle,\,\, \phi[e(i_{u},i_{u,q})]=\langle\langle x, y, 1\rangle\rangle
\end{equation}
for all $1\leq q\neq v\leq m_{u}$,
\begin{equation}\label{uuvuuqb}
\phi[e(i_{p},i_{p,q})]=\phi[b(i_{p})]=\langle\langle x, y, 1\rangle\rangle
\end{equation}
for all $p\in J_{2}'$ and all $q$ with $e(i_{p},i_{p,q})\in B_{2}$ such that $e(i_{u})$ appears in $b(i_{p})$, and $\phi[b]=0$ for all remaining $b\in B_{2}\cup B_{3}$ in the Witt ring of $K((z))$. Therefore, we obtain
$\partial_{z}(\alpha(\phi))=(x, y)\neq 0$. Hence, $\alpha(\phi)$ ramifies.

\textbf{Case 2}: $\exists$ $e(i_{p}, i_{p,q})\in B_{2}'$ with $n_{i_{p}}+n_{i_{p,q}}\geq 3$ such that $i_{p}\in J_{2}'$. Let $e(i_{u}, i_{u,v})\in B_{2}'$ be such an element as in the previous case. Observe that by construction of $B_{3}$ there exists 
\begin{equation}\label{choicekone}
k_{1}\in I\backslash \{i_{p}, j_{r}\,|\, i_{p}\in J_{2}', 1\leq r\leq s\}
\end{equation}
such that $e(k_{1})$ appears in $b'(i_{u})$. We first choose $\phi[e(i_{u,v})]$ as in (\ref{eueuv}) and $\phi[e(k_{1})]=\langle x, y, xy\rangle\perp h$. Then, we choose $\phi_{i}$ for $i\in I\backslash \{i_{u,v}, k_{1} \}$ such that
\begin{equation*}
\phi[e(i)]=\begin{cases}\langle x, y, xy\rangle\perp h & \text{ if } e(k_{1}) \text{ appears in } b(i),\\ \langle 1\rangle\perp h & \text{ otherwise} \end{cases} 
\end{equation*}
for all $i\in \{i_{p}, j_{r}\,|\, i_{p}\in J_{2}', 1\leq r\leq s\}$,
\begin{equation*}
\phi[e(i_{p,q})]=\begin{cases}\langle x, y, xy\rangle\perp h & \text{ if } i_{p}=k_{1} \text{ or } e(k_{1}) \text{ appears in } b(i_{p}),\\ \langle 1\rangle\perp h & \text{ otherwise} \end{cases} 
\end{equation*}
for all $q$ such that $e(i_{p}, i_{p,q})\in B_{2}$, and $\phi_{i}=\langle 1\rangle\perp h$ for the remaining $i\in I\backslash \{i_{u,v}, k_{1} \}$ over $K((z))$. Therefore, we obtain (\ref{uuvuuq}),
\begin{equation}\label{bieipdtypeB}
\phi[b(i)]=\phi[e(i_{p},i_{p,q})]=\langle\langle x, y, 1\rangle\rangle
\end{equation}
for all $i\in \{i_{p}, j_{r}\,|\, i_{p}\in J_{2}', 1\leq r\leq s\}$ such that $e(k_{1})$ appears in $b(i)$ and for all $e(i_{p}, i_{p,q})\in B_{2}$ such that $i_{p}=k_{1}$ or $e(k_{1})$ appears in $b(i_{p})$, and $\phi[b]=0$ for all remaining $b\in B_{2}\cup B_{3}$ in the Witt ring of $K((z))$. Hence, $\partial_{z}(\alpha(\phi))=(x, y)\neq 0$, thus $\alpha(\phi)$ ramifies.\end{proof}

We present the second main result on the group of unramified degree $3$ invariants for type $B$.

\begin{theorem}\label{secthm}
Let $G=(\prod_{i=1}^{m}\gSpin_{2n_{i}+1})/\gmu$ defined over an algebraically closed field $F$, $m, n_{i}\geq 1$, where $\gmu$ is a central subgroup. Then, every unramified degree $3$ invariant of $G$ is trivial, i.e., $\Inv^{3}_{\nr}(G)=0$.
\end{theorem}
\begin{proof}
Let $G_{\red}=(\prod_{i=1}^{m}\gGamma_{2n_{i}+1})/\gmu$. Since the classifying space $BG$ is stably birational to the classifying space $BG_{\red}$, by (\ref{Rostisomrphismm}) we have $\Inv^{3}_{\nr}(G)=\Inv^{3}_{\nr}(G_{\red})$. We shall show that $\Inv^{3}_{\nr}(G_{\red})=0$. Let $G'=(\gSpin_{3})^{m}/\gmu$ and $G'_{\red}=(\gGamma_{3})^{m}/\gmu$. Then, the standard embeddings $\gSpin_{3}\to \gSpin_{2n_{i}+1}$ and $\gGamma_{3}\to \gGamma_{2n_{i}+1}$ induce morphisms $G'\to G$ and $G'_{\red}\to G_{\red}$, thus we have
\begin{equation}\label{seconddiag}
\xymatrix{
	\Inv^{3}(G) \ar@{->}[r] & \Inv^{3}(G')  \\
	\Inv^{3}(G_{\red})  \ar@{^{(}->}[u]\ar@{->}[r] & \Inv^{3}(G'_{\red}) \ar@{^{(}->}[u]\\
}
\end{equation}
By (\ref{formcoreqa}) in Proposition \ref{formcor} and Lemma \ref{lemmaunrami} we may assume that the bottom map in (\ref{seconddiag}) is an isomorphism. By \cite[Lemma 4.3]{Mer17} and the exceptional isomorphism $A_{1}=B_{1}$, we have $\Inv^{3}_{\nr}(G'_{\red})=0$, thus every invariant of $G_{\red}$ is ramified.
\end{proof}

\subsection{Type $C$}

\begin{lemma}\label{honegredpropC}
	Let $G=(\prod_{i=1}^{m}\gSp_{2n_{i}})/\gmu$, $m, n_{i}\geq 1$, where $\gmu$ is a central subgroup. Let $R$ be the subgroup of $(\gmu_{2}^{m})^{*}=(\Z/2\Z)^{m}$ whose quotient is the character group $\gmu^{*}$. Set $G_{\red}=(\prod_{i=1}^{m}\gGSp_{2n_{i}})/\gmu$, where $\gGSp_{2n_{i}}$ denotes the group of symplectic similitudes. Then, for any field extension $K/F$ the first Galois cohomology set $H^{1}(K, G_{\red})$ is bijective to the set of $m$-tuples $\big((A_{1},\sigma_{1}),\ldots, (A_{m},\sigma_{m})\big)$ of pairs of central simple $K$-algebra $A_{i}$ of degree $2n_{i}$ with symplectic involution $\sigma_{i}$ such that for all $r=(r_{i})\in R$
	\begin{equation*}
	r_{1}A_{1}+\cdots +r_{m}A_{m}=0 \text{ in } \Br(K),
	\end{equation*}
	where $\Br(K)$ denotes the Brauer group of $K$.
\end{lemma}
\begin{proof}
	Let $G_{\red}=(\prod_{i=1}^{m}\gGSp_{2n_{i}})/\gmu$, where $\gGSp_{2n_{i}}$ denotes the group of symplectic similitudes. Consider the exact sequence
	\[1\to (\gm)^{m}/\gmu \to G_{\red}\to \prod_{i=1}^{m}\gPGSP_{2n_{i}}\to 1.\]
	Then, by the same argument as in the proof of Lemma \ref{honegredprop} the set $H^{1}(F,G_{\red})$ is bijective to the kernel of following map
	\[H^{1}(F, \prod_{i=1}^{m}\gPGSP_{2n_{i}})\stackrel{}\to \Br_{2}(F)^m\stackrel{\tau}\to H^{2}(F,(\gmu_{2})^{m}/\gmu),\] 
	where the first map sends an $m$-tuple $\big((A_{1},\sigma_{1}),\ldots, (A_{m},\sigma_{m})\big)$ of simple algebra $A_{i}$ of degree $2n_{i}$ with symplectic involution $\sigma_{i}$ to the $m$-tuple $(A_{1}, \ldots, A_{m})$ and the map $\tau$ is induced by the natural surjection $(\gmu_{2})^{m}\to (\gmu_{2})^{m}/\gmu$. Since $(A_{1}, \ldots, A_{m})\in \Ker(\tau)$ if and only if it is contained in the kernel of the map in (\ref{Bdescrip}) for all $r\in R$, thus we have
	\begin{equation}\label{HoneGredC}
	H^{1}(F, G_{\red})\simeq \{\big((A_{1},\sigma_{1}),\ldots, (A_{m},\sigma_{m})\big)\,|\,\deg A_{i}=2n_{i}, \sum_{i=1}^{m}r_{i}A_{i}=0         \}   
	\end{equation}
	for all $r=(r_{i})\in R$.\end{proof}

Let $(A, \sigma)$ be a pair of central simple $F$-algebra $A$ of degree $2n$ with involution $\sigma$ of the first kind. The trace form $T_{\sigma}:A\to F$ is given by $T_{\sigma}(a)=\operatorname{Trd}(\sigma(a)a)$, where $\operatorname{Trd}$ denotes the reduced trace. We denote by $T_{\sigma}^{+}$ the restriction of $T_{\sigma}$ to $\operatorname{Sym}(A,\sigma)$. Set
\begin{equation}\label{phiR}
\phi[r]:=\perp_{i=1}^{m} r_{i}\phi_{i}, \text{ where } \phi_{i}=\begin{cases}
T_{\sigma_{i}} & \text{ if } n_{i}\equiv 1 \mod 2,\\
T_{\sigma_{i}}^{+} & \text{ if } n_{i}\equiv 2 \mod 4\\
\end{cases}
\end{equation}
for all $r=(r_{i})\in R\cap (\bigoplus_{4\nmid n_{i}}(\Z/2\Z) e_{i})$. For all $i\in I$ such that $4|n_{i}$, we simply write $\Delta$ for the Garibaldi-Parimala-Tignol invariant $\Delta(A_{i}, \sigma_{i})$ defined in \cite[Theorem A]{GPT}. Then, this degree $3$ invariant induces the following invariants of $G_{\red}$
\begin{equation}\label{GPTinvarintpro}
\Delta_{i}: H^{1}(K, G_{\red})\to H^{1}(K, \gPGSP_{2n_{i}})\stackrel{\Delta}\to H^{3}(K),
\end{equation}
where the first map in (\ref{GPTinvarintpro}) is the projection and $K/F$ is a field extension. We show that every invariant of semisimple group of type $C$ is generated by the Arason invariants associated to $\phi[r]$ and the Garibaldi-Parimala-Tignol invariants $\Delta_{i}$.

\begin{proposition}\label{formcorC}
	Let $G=(\prod_{i=1}^{m}\gSp_{2n_{i}})/\gmu$ defined over an algebraically closed field $F$, where $m, n_{i}\geq 1$, $\gmu$ is a central subgroup. Let $R$ be the subgroup of $(\gmu_{2}^{m})^{*}$ whose quotient is the character group $\gmu^{*}$. Set $G_{\red}=(\prod_{i=1}^{m}\gGSp_{2n_{i}})/\gmu$. Then, every normalized invariant in $\Inv^{3}(G_{\red})$ is of the form 
	\begin{equation}\label{Cinvariantform}
	\sum_{r\in R'}\boldsymbol{\mathrm{e}}_{3}(\phi[r])+\sum_{i\in I'}\Delta_{i}
	\end{equation}
	for some $R'\subseteq R\cap (\bigoplus_{4\nmid n_{i}}(\Z/2\Z) e_{i})$ and some subset $I'\subseteq \{i\in I\,|\, 4|n_{i}\}$, where $\phi[r]$ denotes the quadratic form defined in $($\ref{phiR}$)$ and $\boldsymbol{\mathrm{e}}_{3}: I^{3}(K)\to H^{3}(K)$ denotes the Arason invariant over a field extension $K/F$. Moreover, we have
	\begin{equation}\label{typeCprosec}
	\Inv^{3}(G_{\red})_{\norm}\simeq \frac{\bigoplus_{4\,| n_{i}}(\Z/2\Z) e_{i}\bigoplus\big(R\cap (\bigoplus_{4\nmid n_{i}}(\Z/2\Z) e_{i})\big)}{\langle e_{i},\, e_{j}+e_{k}\in R\,|\, e_{j}, e_{k}\not\in R,\, n_{j}\equiv n_{k}\equiv 1\mod 2\rangle}.
	\end{equation}
\end{proposition}
\begin{proof}
Since $F$ is algebraically closed, we get $\Inv^{3}(G_{\red})_{\norm}=\Inv^{3}(G_{\red})_{\ind}$. Let $i$ be an integer such that $n_{i}\equiv 0 \mod 4$. If $e_{i}\in R$, then, as every symplectic
involution on a split algebra is hyperbolic, by Lemma \ref{honegredpropC} and \cite[Theorem A]{GPT} the invariant $\Delta_{i}$ defined in (\ref{GPTinvarintpro}) vanishes. Now assume that $e_{i}\not\in R$.  Let $Q=(x, y)$ be a division quaternion algebra over a field extension $K/F$ and let $b=\langle 1, z\rangle\perp h$ be a symmetric bilinear form on $E^{n_{i}}$, where $h$ denotes a hyperbolic form and $E=K((z))$. Consider the linear system as in (\ref{HoneGredC}) with the coefficients given by a basis of $R$. As $e_{i}\not\in R$, it follows by the rank theorem (or Rouch\'e-Capelli theorem) that there exists a $G_{\red}$-torsor $\eta=\big((A_{1},\sigma_{1}),\ldots, (A_{m},\sigma_{m})\big)$ over $E$ such that
\begin{equation}\label{etatypeC}
(A_{i},\sigma_{i})=(M_{n_{i}}(Q), \sigma_{b}\tens \gamma) \text{ and }  (A_{j},\sigma_{j})=(M_{2n_{j}}(E), \sigma_{\omega}) \text{ or } (M_{n_{j}}(Q), t\tens \gamma) 
\end{equation}
for all $1\leq j\neq i\leq m$, where $\gamma$ denotes the canonical involution on $Q$, $\sigma_{b}$ denotes the adjoint involution on $\End(E^{n_{i}})=M_{n_{i}}(E)$ with respect to $b$, $\sigma_{\omega}$ denotes the adjoint involution with respect to the standard symplectic bilinear form $\omega$, and $t$ denotes the transpose involution on $M_{n_{j}}(E)$. Then, by \cite[Example 3.1]{GPT} we have 
\begin{equation}\label{etadeltaai}
\Delta_{i}(\eta)=(Q)\cup (z),
\end{equation}
thus, $\partial_{z}(\alpha(\eta))=(x, y)\neq 0$. Therefore, we have a nontrivial invariant $\Delta_{i}$ of order $2$
for any $i$ such that $n_{i}\equiv 0 \mod 4$ and  $e_{i}\not\in R$.

Let $r\in R\cap (\bigoplus_{4\nmid n_{i}}(\Z/2\Z) e_{i})$. Since each quadratic form $\phi_{i}$ in (\ref{phiR}) has even dimension and trivial discriminant, we obtain $\phi[r]\in I^{2}(K)$ for each $r$. By \cite[Theorem 1]{Que} the Hasse invariant of $\phi_{i}$ in (\ref{phiR}) coincides with the class of $A_{i}$ in $\Br(K)$, thus by the relation in (\ref{HoneGredC}), we have $\phi[r]\in I^{3}(K)$ for each $r\in R\cap (\bigoplus_{4\nmid n_{i}}(\Z/2\Z) e_{i})$. Therefore, the Arason invariant induces a normalized invariant $\boldsymbol{\mathrm{e}}_{3}(\phi[r])$ of order dividing $2$ that sends an $m$-tuple in (\ref{HoneGredC}) to $\boldsymbol{\mathrm{e}}_{3}(\phi[r])\in H^{3}(K)$.

Let $r\in R_{1}''+R_{2}''$, where $R_{1}''=\langle e_{i}\in R\,|\, n_{i}\not\equiv 0 \mod 4\rangle$ and $R_{2}''$ denotes the subgroup of $R$ defined in Proposition \ref{decGC}. For any $e_{i}\in R_{1}''$ and any $e_{j}+e_{k}\in R_{2}''$, we have 
\begin{equation}\label{tijvanish}
\phi_{i}=T_{\sigma_{i}}=h\, \text{ and }\, \phi_{j}\perp \phi_{k}=T_{\sigma_{j}}\perp T_{\sigma_{k}}=\langle\langle a, b, 1\rangle\rangle\perp h',
\end{equation}  
where $A_{j}=A_{k}=(a, b)$ in $\Br(K)$, $h$ and $h'$ denote hyperbolic forms, thus both invariants $\boldsymbol{\mathrm{e}}_{3}(\phi[e_{i}])$ and $\boldsymbol{\mathrm{e}}_{3}(\phi[e_{j}+e_{k}])$ vanish. Therefore, the invariant $\boldsymbol{\mathrm{e}}_{3}(\phi[r])$ vanishes.

To complete the proof, by Theorem \ref{mainthmC} it suffices to show that the invariant $\boldsymbol{\mathrm{e}}_{3}(\phi[r])$ is nontrivial for any $r\in R\cap (\bigoplus_{4\nmid n_{i}}(\Z/2\Z) e_{i})\backslash (R_{1}''+R_{2}'')$. Let $G'_{\red}=(\gGSp_{2})^{m}/\gmu$. Then, the rest of the proof of Proposition \ref{formcor} still works if we replace the exceptional isomorphism $A_{1}=B_{1}$, the standard embedding $\gGamma_{3}\to \gGamma_{2n_{i}+1}$, and Lemma \ref{lemmaunrami} in the proof of Proposition \ref{formcor} by the exceptional isomorphism $A_{1}=C_{1}$, the standard embedding $\gGSp_{2}\to \gGSp_{2n_{i}}$, and Lemma \ref{lemmaunramiC}, respectively.\end{proof}

\begin{remark}
If $m=2$, $n_{1}\equiv n_{2}\equiv 0 \mod 2$, and $\gmu\subseteq \gmu_{2}^{2}$ is the diagonal subgroup, then the invariant in Proposition \ref{formcorC} coincides with the invariant defined in \cite{BMT}.
\end{remark}

We present the following analogue of Lemma \ref{lemmaunrami}, which plays the same role for the triviality of unramified invariants as Lemma \ref{lemmaunrami} plays for the groups of type $B$.
\begin{lemma}\label{lemmaunramiC}
	Let $G=(\prod_{i=1}^{m}\gSp_{2n_{i}})/\gmu$ defined over an algebraically closed field $F$, where $m, n_{i}\geq 1$, $\gmu$ is a central subgroup. Set $G_{\red}=(\prod_{i=1}^{m}\gGSp_{2n_{i}})/\gmu$. Then, every normalized invariant in $\Inv^{3}(G_{\red})$ is ramified if either $n_{i}$ is divisible by $4$ for some $i$ with $e_{i}\not\in R_{1}$ or $n_{j}n_{k}\not\equiv 1 \mod 2$ for some $j$ and $k$ such that $e_{j}+e_{k}\in R\cap (\bigoplus_{4\nmid n_{i}}(\Z/2\Z) e_{i})$.
\end{lemma}
\begin{proof}
Let $\alpha$ be a normalized invariant in $\Inv^{3}(G_{\red})$ be written as in (\ref{Cinvariantform}) for some subspace $R'\subseteq R\cap (\bigoplus_{4\nmid n_{i}}(\Z/2\Z) e_{i})$ and subset $I'\subseteq \{i\in I\,|\, n_{i}\equiv 0 \mod 4,  e_{i}\not\in R\}$.

Assume that there exist $i\in I'$. Let $\eta=\big((A_{1},\sigma_{1}),\ldots, (A_{m},\sigma_{m})\big)$ be a $G_{\red}$-torsor as in the proof of Proposition \ref{formcorC}. Then, by (\ref{etatypeC}), \cite[Example 3.1]{GPT}, and \cite[Theorem A]{GPT} we have 
\[\Delta_{j}(\eta)=0\]
for all $j\neq i$ such that $n_{j}\equiv 0 \mod 4$. Since
\[\phi_{j}=\begin{cases} h & \text{ if } (A_{j},\sigma_{j})=(M_{2n_{j}}(E), \sigma_{\omega}),\\ \langle\langle x, y\rangle\rangle\perp h & \text{ if } (A_{j},\sigma_{j})=(M_{n_{j}}(Q), t\tens \gamma), \end{cases}\]
where $h$ denotes a hyperbolic form and the pairs of the form $(M_{n_{j}}(Q), t\tens \gamma)$ appear an even number of times in the relation of (\ref{HoneGredC}) for any $r\in R'$, we have $\boldsymbol{\mathrm{e}}_{3}(\phi[r])=0$ for any $r\in R'$. Therefore, by (\ref{etadeltaai}) we have $\partial_{z}(\alpha(\eta))=(x, y)\neq 0$, thus the invariant $\alpha$ ramifies.

We may assume that $n_{i}\not\equiv 0 \mod 4$ for all $1\leq i\leq m$, thus
\[\alpha(\eta)=\boldsymbol{\mathrm{e}}_{3}(\phi[r_{2}])+\boldsymbol{\mathrm{e}}_{3}(\phi[r_{3}])  \]
for some nonzero $r_{2}\in R_{2}$ and some $r_{3}\in R_{3}$, where $R_{1}$ and $R_{2}$ denote the subspaces of $R$ in (\ref{RoneRtwo}), $R_{3}$ is a complementary subspace of $R_{1}+R_{2}$ in $R$, and $\eta$ is a $G_{\red}$-torsor. Then, we choose bases $B_{2}=\{e(i_{p}, i_{p,q})\}$ of $R_{2}$ such that either $n_{i_{p,q}}$ is even or $n_{i_{p}}+n_{i_{p,q}}$ is even, and $B_{3}$ of a complementary subspace of $R_{2}$ as in Lemma \ref{lemmaunrami} so that the invariant $\alpha$ is written as in (\ref{alphaphire}). We show that the invariant $\alpha(\eta)$ ramifies following the proof of Lemma \ref{lemmaunrami}.

\textbf{Case 1}: $\exists$ $e(i_{p}, i_{p,q})\in B_{2}'$ with $n_{i_{p}}n_{i_{p,q}}\not\equiv 1 \mod 2$ such that $i_{p}\not\in J_{2}'$. Let $e(i_{u}, i_{u,v})\in B_{2}'$ be such an element for some $1\leq u\leq k$ and $1\leq v\leq m_{u}$. Let $Q=(x, y)$ be a division quaternion algebra over $K/F$ and let $Q_{1}=(x, z)$ and $Q_{2}=(x, yz)$ be quaternions over $E$. We denote by $\gamma$, $\gamma_{1}$, $\gamma_{2}$ the canonical involutions on $Q$, $Q_{1}$, $Q_{2}$, respectively. For the sake of simplicity, we shall write the symbol $d$ for the corresponding degree of the matrix algebras in the rest of the proof. Now we choose $\eta=\big((A_{i},\sigma_{i})\big)$ for $i\in I$ such that
\begin{equation}\label{chocaseoneC}
(A_{i}, \sigma_{i})=(M_{d}(Q), t\tens \gamma),\,\, (A_{i_{u,v}}, \sigma_{i_{u,v}})=(M_{d}(Q_{1}\tens Q_{2}), t\tens \gamma_{1}'\tens \gamma_{2})
\end{equation}
for $i=i_{u}, i_{u,q}$ and all $1\leq q\neq v\leq m_{u}$, where $t$ denotes the transpose involution on a matrix algebra and $\gamma_{1}'$ denotes an orthogonal involution on $Q_{1}$ given by the composition of $\gamma_{1}$ and the inner automorphism induced by one of the generators of pure quaternions in $Q_{1}$,
\begin{equation}\label{chooseCtwo}
(A_{i_{p}}, \sigma_{i_{p}}), (A_{i_{p,q}}, \sigma_{i_{p,q}})=\begin{cases}(M_{d}(Q), t\tens \gamma) & \text{ if } e(i_{u}) \text{ appears in } b(i_{p}),\\ (M_{d}(E), \sigma_{\omega}) & \text{ otherwise,} \end{cases} 
\end{equation}
for all $i_{p}\in J_{2}'$ and all $q$ with $e(i_{p},i_{p,q})\in B_{2}$, and 
\begin{equation}\label{chooseCthree}
(A_{i}, \sigma_{i})=(M_{d}(E), \sigma_{\omega})
\end{equation}
for the remaining $i\in I$. Then, we have
\begin{equation*}
\phi[e(i_{u})]=\phi[e(i_{u,q})]=\langle\langle x, y\rangle\rangle\perp h,\,\, \phi[e(i_{u,v})]=\langle z, xz, yz, xyz\rangle\perp h
\end{equation*}
for all $1\leq q\neq v\leq m_{u}$, thus we obtain (\ref{uuvuuq}), (\ref{uuvuuqb}), and $\phi[b]=0$ for all remaining $b\in B_{2}\cup B_{3}$ in the Witt ring of $E$. Hence, $\partial_{z}(\alpha(\eta))=(x, y)\neq 0$, i.e., $\alpha$ ramifies.

\textbf{Case 2}:  $\exists$ $e(i_{p}, i_{p,q})\in B_{2}'$ with $n_{i_{p}}n_{i_{p,q}}\not\equiv 1 \mod 2$ such that $i_{p}\in J_{2}'$. Let $e(i_{u}, i_{u,v})\in B_{2}'$ be such an element. We choose $k_{1}$ as in (\ref{choicekone}) and then choose $(A_{k_{1}}, \sigma_{k_{1}})$ and $(A_{i_{u,v}}, \sigma_{i_{u,v}})$ as in (\ref{chocaseoneC}). Then, we choose $(A_{i}, \sigma_{i})$ for $i\in I\backslash \{i_{u,v}, k_{1} \}$ such that
\begin{equation}\label{typeCcasetwoone}
(A_{i}, \sigma_{i})=\begin{cases}(M_{d}(Q), t\tens \gamma) & \text{ if } e(k_{1}) \text{ appears in } b(i),\\ (M_{d}(E), \sigma_{\omega}) & \text{ otherwise} \end{cases} 
\end{equation}
for all $i\in \{i_{p}, j_{r}\,|\, i_{p}\in J_{2}', 1\leq r\leq s\}$,
\begin{equation}\label{typeCcasetwooneprime}
(A_{i_{pq}}, \sigma_{i_{p,q}})=\begin{cases}(M_{d}(Q), t\tens \gamma) & \text{ if } i_{p}=k_{1} \text{ or } e(k_{1}) \text{ appears in } b(i_{p}),\\ (M_{d}(E), \sigma_{\omega}) & \text{ otherwise} \end{cases} 
\end{equation}
for all $q$ such that $e(i_{p}, i_{p,q})\in B_{2}$, and
\begin{equation}\label{typeCcasetwoano}
(A_{i}, \sigma_{i})=(M_{d}(E), \sigma_{\omega})
\end{equation}
for the remaining $i\in I\backslash \{i_{u,v}, k_{1} \}$. Therefore, we obtain (\ref{uuvuuq}), (\ref{bieipdtypeB}), and $\phi[b]=0$ for all remaining $b\in B_{2}\cup B_{3}$ in the Witt ring of $E$. Therefore, $\partial_{z}(\alpha(\eta))=(x, y)\neq 0$, thus $\alpha$ ramifies.\end{proof}

We show that the same result in Theorem \ref{secthm} holds for the groups of type $C$.

\begin{theorem}\label{secthmC}
	Let $G=(\prod_{i=1}^{m}\gSp_{2n_{i}})/\gmu$ defined over an algebraically closed field $F$, $m, n_{i}\geq 1$, where $\gmu$ is a central subgroup. Then, every unramified degree $3$ invariant of $G$ is trivial, i.e., $\Inv^{3}_{\nr}(G)=0$.
\end{theorem}
\begin{proof}
Let $G_{\red}=(\prod_{i=1}^{m}\gGSp_{2n_{i}})/\gmu$, $G'_{\red}=(\gGSp_{2})^{m}/\gmu$, and $G'=(\gSp_{2})^{m}/\gmu$. Then, the proof of Theorem \ref{secthm} still works if we replace Proposition \ref{formcor}, Lemma \ref{lemmaunrami}, and the exceptional isomorphism $A_{1}=B_{1}$ in the proof by Proposition \ref{formcorC}, Lemma \ref{lemmaunramiC}, and the exceptional isomorphism $A_{1}=C_{1}$, respectively.
\end{proof}

\subsection{Type $D$}

\begin{lemma}\label{honegredpropD}
	Let $G=(\prod_{i=1}^{m}\gSpin_{2n_{i}})/\gmu$, $m, \geq 1$, $n_{i}\geq 3$, where $\gmu$ is a central subgroup.Let $R$ be the subgroup of the character group $Z$ defined in $($\ref{centerD}$)$ such that $\gmu^{*}=Z/R$. Set $G_{\red}=(\prod_{i=1}^{m}\operatorname{\mathbf{\Omega}}_{2n_{i}})/\gmu$, where $\operatorname{\mathbf{\Omega}}_{2n_{i}}$ denotes the extended Clifford group. Then, for any field extension $K/F$ the first Galois cohomology set $H^{1}(K, G_{\red})$ is bijective to the set of $m$-tuples $\big((A_{1},\sigma_{1}, f_{1}),\ldots, (A_{m},\sigma_{m}, f_{m})\big)$ of triples consisting of a central simple $K$-algebra $A_{i}$ of degree $2n_{i}$ with orthogonal involution $\sigma_{i}$ of trivial discriminant and a $K$-algebra isomorphism $f_{i}: Z(C(A_{i}, \sigma_{i})\big)\simeq K\times K$, where $Z\big(C(A_{i}, \sigma_{i})\big)$ denotes the center of the Clifford algebra $C(A_{i},\sigma_{i})$, satisfying 
	\begin{equation*}\label{conditionD}
		B_{1}+\cdots +B_{m}=0 \text{ in } \Br(K) 
	\end{equation*}
	for all $\sum_{i=1}^{m}r'_{i}\in R$  with \[r'_{i}=\begin{cases} r_{i}e_{i} & \text{ if } n_{i} \text{ odd,}\\ r_{i,1}e_{i,1}+r_{i,2}e_{i,2} & \text{ if } n_{i} \text{ even},\end{cases}\]
	where \[B_{i}:=\begin{cases} r_{i}C_{i,1}\, \text{ or }\, r_{i}C_{i,2} & \text{ if } n_{i} \text{ odd,}\\ r_{i,1}C_{i,1}+r_{i,2}C_{i,2} \,\text{ or }\, r_{i,1}C_{i,2}+r_{i,2}C_{i,1} & \text{ if } n_{i} \text{ even}\end{cases}\]
	depending on the choice of two isomorphisms $f_{i}$ for each triple $(A_{i},\sigma_{i}, f_{i})$, $C_{i,1}$ and $C_{i,2}$ denote simple $K$-algebras such  that $C(A_{i}, \sigma_{i})=C_{i,1}\times C_{i,2}$, and $\Br(K)$ denotes the Brauer group of $K$.
\end{lemma}
\begin{proof}
	Let $G_{\red}=(\prod_{i=1}^{m}\operatorname{\mathbf{\Omega}}_{2n_{i}})/\gmu$, where $\operatorname{\mathbf{\Omega}}_{2n_{i}}$ denotes the extended Clifford group (\cite[\S 13]{KMRT}). Consider the exact sequence
	\[1\to (\gm)^{2m}/\gmu \to G_{\red}\to \prod_{i=1}^{m}\gPGO^{+}_{2n_{i}}\to 1,\]
	where $\gPGO^{+}_{2n_{i}}$ denotes the projective orthogonal group. Applying the same argument as in the proof of Lemma \ref{honegredprop} we see that the set $H^{1}(K,G_{\red})$ is bijective to the kernel of following map
	\[H^{1}(K, \prod_{i=1}^{m}\gPGO^{+}_{2n_{i}})\stackrel{\beta}\to \Br\big(Z(\prod_{i=1}^{m}\gSpin_{2n_{i}})\big)\stackrel{\tau}\to H^{2}(K,Z(\prod_{i=1}^{m}\gSpin_{2n_{i}})/\gmu),\] 
	where the map $\beta$ sends an $m$-tuple $\big((A_{i},\sigma_{i}, f_{i})\big)$ of triples consisting of a central simple $K$-algebra $A_{i}$ of degree $2n_{i}$ with orthogonal involution $\sigma_{i}$ of trivial discriminant and a $K$-algebra isomorphism $f_{i}: Z(C(A_{i}, \sigma_{i})\big)\simeq K\times K$ to the $m$-tuple $(B_{1}', \ldots, B_{m}')$ with  \[B_{i}':=\begin{cases} C_{i,1} \, \text{ or }\, C_{i,2} & \text{ if } n_{i} \text{ odd,}\\ (C_{i,1}, C_{i,2})  \,\text{ or }\, (C_{i,2}, C_{i,1}) & \text{ if } n_{i} \text{ even},\end{cases}\]
	depending on the choice of two isomorphisms $f_{i}$ for each triple $(A_{i},\sigma_{i}, f_{i})$ (i.e., For odd (resp. even) $n_{i}$, the image of $(A_{i},\sigma_{i}, f_{i})$  under $\beta$ is $C_{i,1}$ (resp. $(C_{i,1}, C_{i,2})$) if and only if the image of $(A_{i},\sigma_{i}, f'_{i})$ for another isomorphism $f'_{i}: Z(C(A_{i}, \sigma_{i})\big)\simeq K\times K$ under $\beta$ is $C_{i,2}$ (resp. $(C_{i,2}, C_{i,1})$)) and the map $\tau$ is induced by the natural surjection $Z(\prod_{i=1}^{m}\gSpin_{2n_{i}})\to Z(\prod_{i=1}^{m}\gSpin_{2n_{i}})/\gmu$. As $(B'_{1}, \ldots, B'_{m})\in \Ker(\tau)$ if and only if it is contained in the kernel of the composition
	 \begin{equation*}
	 H^{2}\big(K, Z(\prod_{i=1}^{m}\gSpin_{2n_{i}})\big)\stackrel{\tau}\to H^{2}\big(K,Z(\prod_{i=1}^{m}\gSpin_{2n_{i}})/\gmu\big)\stackrel{r_{*}}\to H^{2}(K, \gm)
	 \end{equation*}
	 for all $r\in R=(Z(\prod_{i=1}^{m}\gSpin_{2n_{i}})/\gmu)^{*}$, we obtain
	 \begin{equation}\label{HoneGredD}
	 H^{1}(K, G_{\red})\simeq \{\big((A_{i},\sigma_{i}, f_{i})\big)\,|\, \sum_{i=1}^{m}B_{i}=0 \text{ in } \Br(K)    \}   
	 \end{equation}
	 for all $\sum_{i=1}^{m}r'_{i}\in R$.\end{proof}

Recall from Theorem \ref{mainthmD} the following subsets
\begin{align*}
&I_{1}=\{i\,|\, Z_{i}=R_{1,i} \text{ or } R_{1,i}', n_{i}\neq 3\}\, \text{ and }\\
&I_{2}=\{i \,|\,R_{1,i}'=0,  4|n_{i}\}\cup \{i\,|\, R_{1,i}'=(\Z/2\Z)e_{i,1} \text{ or } (\Z/2\Z)e_{i,2}, n_{i}\geq 6, 4|n_{i}\}=:I_{21}\cup I_{22}.
\end{align*}
Let $i\in I_{1}$. Then, from Lemma \ref{honegredpropD}, we see that both $K$-algebras $A_{i}$ and $C(A_{i},\sigma_{i})$ split, thus we have $(A_{i},\sigma_{i},f_{i})\simeq (M_{2n_{i}}(K),\sigma_{\psi_{i}})$ for some adjoint involution $\sigma_{\psi_{i}}$ with respect to a quadratic form $\psi_{i}$ such that $\psi_{i}\in I^{3}(K)$. Hence, the Arason invariant $\boldsymbol{\mathrm{e}_{3}}$ induces the following invariant 
\begin{equation}\label{arasonethree}
\boldsymbol{\mathrm{e}}_{3, i}: H^{1}(K, G_{\red})\to H^{3}(K)
\end{equation}	
given by $\boldsymbol{\mathrm{e}}_{3, i}\big((A_{1},\sigma_{1}, f_{1}),\ldots, (A_{m},\sigma_{m}, f_{m})\big)=\boldsymbol{\mathrm{e}_{3}}(\psi_{i})$. This invariant is obviously nontrivial.

Now let $i\in I_{2}$. Then, the invariant $\Delta'$ of $\gPGO^{+}_{2n}$ (\cite[Theorem 4.7]{Mer163}) gives the following invariant of $G_{\red}$
\begin{equation}\label{merpgo}
\Delta_{i}': \begin{cases} H^{1}(K, G_{\red})\to H^{1}(K, \gPGO^{+}_{2n_{i}})\stackrel{\Delta'}\to H^{3}(K) & \text{ if } i\in I_{21},\\ H^{1}(K, G_{\red})\to H^{1}(K, \gHSpin_{2n_{i}})\to H^{1}(K, \gPGO^{+}_{2n_{i}})\stackrel{\Delta'}\to H^{3}(K) & \text{ if } i\in I_{22},\end{cases}
\end{equation}
where $\gHSpin_{2n_{i}}$ denotes the half-spin group and the first map in (\ref{merpgo}) is the projection for each case.

We shall need the following analogue of \cite[Example 3.1]{GPT}.

\begin{lemma}\label{lemmaforortho}
Let $Q$ be a quaternion algebra over $F$ and let $(A,\sigma, f)\in H^{1}(F,\gPGO^{+}_{2n})$ such that $n\equiv 0 \mod 2$ and $(A,\sigma)=(M_{n}(F)\tens Q, \sigma_{1}\tens \sigma_{2})$
for some orthogonal involutions $\sigma_{1}$ and $\sigma_{2}$ on $M_{n}(F)$ and $Q$, respectively. Then, we have $$\Delta'(A,\sigma, f)=Q\cup (\disc \sigma_{1}).$$
\end{lemma}
\begin{proof}
Let $t$ be the transpose involution on $M_{n}(F)$. Since $\sigma_{1}=\operatorname{Int}(x)\circ t$ for some $t$-symmetric invertible element $x$, where $\operatorname{Int}(x)$ denotes the inner automorphism induced by $x$, we have
\[ \disc(\sigma_{1})=\Nrd_{M_{n}(F)}(x)=\sqrt{\Nrd_{A}(x\tens  1)}    \]
and $\sigma=\operatorname{Int}(x\tens 1)\circ (t\tens \sigma_{2})$, where $\Nrd$ denotes the reduced norm. As $x\tens 1$ is a $\sigma$-symmetric invertible element, the result follows from \cite[\S 4b]{Mer163}.
\end{proof}

\begin{proposition}\label{formcorD}
	Let $G=(\prod_{i=1}^{m}\gSpin_{2n_{i}})/\gmu$ defined over an algebraically closed field $F$, where $m \geq 1$, $n_{i}\geq 3$, $\gmu$ is a central subgroup. Set $G_{\red}=(\prod_{i=1}^{m}\operatorname{\mathbf{\Omega}}_{2n_{i}})/\gmu$, where $\operatorname{\mathbf{\Omega}}_{2n_{i}}$ is the extended Clifford group. Then, every normalized invariant in $\Inv^{3}(G_{\red})$ is of the form 
	\begin{equation}\label{DinvariantformD}
	\sum_{i\in I_{1}'}\boldsymbol{\mathrm{e}}_{3, i}+\sum_{i\in I_{2}'}\Delta_{i}'+\sum_{r\in R''}\boldsymbol{\mathrm{e}}_{3}(\phi[r])
	\end{equation}
	for some subsets $I_{1}'\subseteq I_{1}$, $I_{2}'\subseteq I_{2}$, and $R''\subseteq R'$, where $R'$ denotes the group as defined in Theorem \ref{mainthmD}, $\phi[r]$ is the quadratic form defined in $($\ref{phiR}$)$ and $\boldsymbol{\mathrm{e}}_{3}: I^{3}(K)\to H^{3}(K)$ denotes the Arason invariant for a field extension $K/F$. Moreover, we have
	\begin{equation*}
	\Inv^{3}(G_{\red})_{\norm}\simeq \frac{\bigoplus_{i\in I_{1}\cup I_{2}}(\Z/2\Z) \bar{e}_{i}\bigoplus R'}{\langle \bar{e}_{i},\, \bar{e}_{j}+\bar{e}_{k}\in R'\,|\, \bar{e}_{j}, \bar{e}_{k}\not\in R',\, n_{j}\equiv n_{k}\equiv 1\mod 2\rangle}.
	\end{equation*}
\end{proposition}
\begin{proof}
Since $F$ is algebraically closed, we obtain $\Inv^{3}(G_{\red})_{\norm}=\Inv^{3}(G_{\red})_{\ind}$. We first show that the invariant $\Delta_{j}'$ is nontrivial for all $j\in I_{2}$. Choose a field extension $K/F$ containing variables $x_{i,1}$, $x_{i,2}$, $x_{i}$, $y_{i}$, division quaternion $K$-algebras
\[Q_{i,1}=(x_{i,1}, y_{i}),\, Q_{i,2}=(x_{i,2}, y_{i}) \]
for all $i\in I$ such that $n_{i}\equiv 0 \mod 2$, and cyclic division $K$-algebras 
\[P_{i}=(x_{i}, y_{i})_{4}\]
of exponent $4$ for all $i\in I$ such that $n_{i}\equiv 1 \mod 2$. Let 
\[Q_{i}=\begin{cases}(x_{i,1}x_{i,2},\, y_{i}) & \text{ if } n_{i} \text{ even,}\\ (x_{i},\, y_{i}) & \text{ if } n_{i} \text{ odd},\end{cases} \text{ so that } Q_{i}=\begin{cases}Q_{i,1}+Q_{i,2} & \text{ if } n_{i} \text{ even,}\\ 2P_{i} & \text{ if } n_{i} \text{ odd}\end{cases} \]
in $\Br(K)$. For $r\in R$, let
\[D_{1,r}=\bigotimes_{2|n_{i}}(Q_{i,1}^{r_{i,1}}\tens Q_{i,2}^{r_{i,2}}),\,\, D_{2,r}=\bigotimes_{2\nmid n_{i}}P_{i}^{r_{i}}, \text{ and } D_{r}=D_{1,r}\tens D_{2,r}.\]
Let $L$ be the function field of the product $\prod_{r\in R}\SB(D_{r})$ of Severi-Brauer varieties $\SB(D_{r})$ of $D_{r}$ over $K$. For all $i$ such that $n_{i}\equiv 1 \mod 2$, consider the exterior square $\lambda^{2}P_{i}$ of $P_{i}$ with its canonical involution $\rho_{i}$ \cite[\S 10]{KMRT}. By the exceptional isomorphism $A_{3}=D_{3}$ (\cite[15.32]{KMRT}) we have
\begin{equation}\label{exteriorsquare}
C(\lambda^{2}P_{i}, \rho_{i})=P_{i}\times P_{i}^{\op},
\end{equation}
where $P_{i}^{\op}$ denotes the opposite algebra of $P_{i}$. Let $\chi_{i}$ be a skew-hermitian form over $Q_{i}$ such that $(M_{3}(Q_{i}), \sigma_{\chi_{i}})=(\lambda^{2}P_{i}, \rho_{i})$, where $\sigma_{\chi_{i}}$ is the adjoint involution with respect to $\chi_{i}$. Let $\psi_{i}=\chi_{i}\perp h$ be a skew-hermitian form over $Q_{i}$ of rank $n_{i}$, where $h$ denotes a hyperbolic form (if $n_i=3$, then $\psi_{i}=\chi_{i}$). We denote by $\sigma_{\psi_{i}}$ the adjoint involution on $M_{n_{i}}(Q_{i})$ with respect to $\psi_{i}$. Let
\[(A_{i}, \sigma_{i})=\begin{cases} (M_{n_{i}}(L)\tens Q_{i}, \sigma_{i,1}\tens \sigma_{i,2} ) & \text{ if } n_{i} \text{ even,}\\ (M_{n_{i}}(Q_{i}), \sigma_{\psi_{i}} ) & \text{ if } n_{i} \text{ odd}\end{cases} \]
for some orthogonal involutions $\sigma_{i,1}$ on $M_{n_{i}}(L)$ and $\sigma_{i,2}$ on $Q_{i}$ such that $\disc(\sigma_{i,1})=x_{i,1}$ and $\disc(\sigma_{i,2})=y_{i}$. Then, by \cite[Theorem 1.1]{Tao} and \cite[Corollary 3]{DLT} together with (\ref{exteriorsquare}) we obtain
\[C(A_{i},\sigma_{i})=\begin{cases} M_{2^{n_{i}-2}}(Q_{i,1})\times M_{2^{n_{i}-2}}(Q_{i,2}) & \text{ if } n_{i} \text{ even,}\\ M_{2^{n_{i}-3}}(P_{i})\times M_{2^{n_{i}-3}}(P_{i})^{\op} & \text{ if } n_{i} \text{ odd},\end{cases} \]
thus by a theorem of Amitsur we have a $G_{\red}(L)$-torsor $\eta=\big((A_{i},\sigma_{i}, f_{i})\big)$. Finally, by Lemma \ref{lemmaforortho} we get $\Delta_{j}'(\eta)=(x_{j,1}, x_{j,2}, y_{j})$. As $j\in I_{2}$, we have either $\partial_{x_{j,1}}(\Delta_{j}'(\eta))=(x_{j,2}, y_{j})\neq 0$ or $\partial_{x_{j,2}}(\Delta_{j}'(\eta))=(x_{j,1}, y_{j})\neq 0$, thus $\Delta_{j}'(\eta)\neq 0$.

Now, let $r=(\bar{r}_{1},\ldots, \bar{r}_{m})\in R'$. Then, from Lemma \ref{honegredpropD} we have
\[B_{i}=A_{i}=\begin{cases} 2\bar{r}_{i}C_{i,1}=2\bar{r}_{i}C_{i,2} & \text{ if }n_{i} \text{ odd},\\
\bar{r}_{i}C_{i,1}+\bar{r}_{i}C_{i,2} & \text{ if } n_{i} \text{ even},\end{cases}
\]
in $\Br(K)$, thus the relation in (\ref{HoneGredD}) is equivalent to 
\begin{equation}\label{typeDaiam}
\bar{r}_{1}A_{1}+\cdots +\bar{r}_{m}A_{m}=0.
\end{equation}

As each quadratic form $\phi_{i}$ in (\ref{phiR}) has even dimension and trivial discriminant, we have $\phi[r]\in I^{2}(K)$ for each $r\in R'$. By \cite[Theorem 1]{Que} the Hasse invariant of $\phi_{i}$ in (\ref{phiR}) coincides with the class of $A_{i}$ in $\Br(K)$, thus by the relation in (\ref{typeDaiam}), we have $\phi[r]\in I^{3}(K)$ for each $r\in R'$. Therefore, the Arason invariant induces a normalized invariant $\boldsymbol{\mathrm{e}}_{3}(\phi[r])$ of order dividing $2$ that sends an $m$-tuple in (\ref{HoneGredD}) to $\boldsymbol{\mathrm{e}}_{3}(\phi[r])\in H^{3}(K)$.

Let $r\in \bar{R}_{1}''+\bar{R}_{2}''$, where $\bar{R}_{1}''=\langle \bar{e}_{i}\in R'\rangle$ and $\bar{R}_{2}''=\langle \bar{e}_{j}+\bar{e}_{k}\in R'\,|\, \bar{e}_{j}, \bar{e}_{k}\not\in R', n_{j}\equiv n_{k}\equiv 1 \mod 2\rangle$. Then, by (\ref{tijvanish}) both invariants $\boldsymbol{\mathrm{e}}_{3}(\phi[\bar{e}_{i}])$ and $\boldsymbol{\mathrm{e}}_{3}(\phi[\bar{e}_{j}+\bar{e}_{k}])$ vanish for any $\bar{e}_{i}\in \bar{R}_{1}''$ and any $\bar{e}_{j}+\bar{e}_{k}\in \bar{R}_{2}''$, thus $\boldsymbol{\mathrm{e}}_{3}(\phi[r])$ vanishes.

As before, by Theorem \ref{mainthmD} it is enough to show that the invariant $\boldsymbol{\mathrm{e}}_{3}(\phi[r])$ is nontrivial for any $r\in R'\backslash (\bar{R}_{1}''+\bar{R}_{2}'')$. Let $G'_{\red}=(\prod_{i=1}^{m}G'_{i})/\gmu$, where 
\[G_{i}'=\begin{cases} \gGL_{4} & \text{ if } n_{i}\equiv 1 \mod 2,\\
\gGL_{2}\times \gGL_{2} & \text{ if } n_{i}\equiv 2 \mod 4\end{cases}
\]
and let $r=(\bar{r}_{1},\ldots, \bar{r}_{m})\in R'\backslash (\bar{R}_{1}''+\bar{R}_{2}'')$. Then, by the proof of \cite[Proposition 4.1]{Mer17}, there exists a $G'_{\red}(L')$-torsor $\eta':=(A_{i}')$ for some power series field $L':=K((z))$, where 
\[A_{i}'=\begin{cases} P_{i}' & \text{ if } n_{i}\equiv 1 \mod 2,\\
(Q_{i1}', Q_{i2}') & \text{ if } n_{i}\equiv 2 \mod 4\end{cases}
\]
for some simple algebras $P_{i}'$ of degree $4$ and quaternion algebras $Q_{i1}'$ and $Q_{i2}'$, such that $\partial_{z}\big(\boldsymbol{\mathrm{e}}_{3}(\theta[r])(\eta')\big)\neq 0$, where 
\[\theta[r]:=\perp_{i=1}^{m}\bar{r}_{i}\theta_{i},\,\, \theta_{i}=\begin{cases} T_{\gamma_{i}} & \text{ if } n_{i}\equiv 1 \mod 2,\\
T_{\gamma_{i1}}\perp T_{\gamma_{i1}}  & \text{ if } n_{i}\equiv 2 \mod 4,\end{cases}
\]
$\gamma_{i1}$ and $\gamma_{i2}$ denote the canonical involutions on $Q_{i1}'$ and $Q_{i2}'$, and $\gamma_{i}$ denotes the canonical involution on the quaternion algebra $Q_{i}'$ representing the class of $(P_{i}')^{\tens 2}$.

Now we shall find a $G_{\red}$-torsor $\eta$ such that the value of the invariant $\boldsymbol{\mathrm{e}}_{3}(\phi[r])$ at $\eta$ coincides with the value of the invariant $\boldsymbol{\mathrm{e}}_{3}(\theta[r])$ at $\eta'$. Let 
\[(A_{i}, \sigma_{i})=\begin{cases} \big(M_{n_{i}}(Q_{i}'), \sigma_{\psi_{i}}  \big) & \text{ if } n_{i}\equiv 1 \mod 2,\\
\big(M_{n_{i}/2}(Q_{i1}'\tens Q_{i2}'), t\tens \gamma_{i1}\tens \gamma_{i2} \big) & \text{ if } n_{i}\equiv 2 \mod 4,\end{cases}
\]
where $\sigma_{\psi_{i}}$ denotes the adjoint involution on $M_{n_{i}}(Q_{i}')$ with respect to $\psi_{i}$ as above and $t$ denotes the transpose involution on $M_{n_{i}/2}(L')$. Then, as $\eta'$ is a $G_{\red}'(L')$-torsor and 
\[C(A_{i},\sigma_{i})=\begin{cases}M_{2^{n_{i}-3}}(P_{i}')\times M_{2^{n_{i}-3}}(P_{i}')^{\op} & \text{ if } n_{i}\equiv 1 \mod 2,\\ M_{2^{n_{i}-2}}(Q_{i,1}')\times M_{2^{n_{i}-2}}(Q_{i,2}') & \text{ if } n_{i}\equiv 2 \mod 4, \end{cases} \]
we obtain a $G_{\red}(L')$-torsor $\eta:=\big((A_{i},\sigma_{i},f_{i})\big)$.

Let $T^{-}_{\gamma_{i1}\tens \gamma_{i2}}$ denote the restriction of $T_{\gamma_{i1}\tens \gamma_{i2}}$ to $\operatorname{Skew}(Q_{i1}'\tens Q_{i2}', \gamma_{i1}\tens \gamma_{i2})$. By a direct calculation we have
\[\phi_{i}=\begin{cases} T_{\gamma_{i}}\perp h & \text{ if } n_{i}\equiv 1 \mod 2,\\
T^{+}_{\gamma_{i1}\tens \gamma_{i2}}\perp h\,\, (\text{resp.}\,\, T^{-}_{\gamma_{i1}\tens \gamma_{i2}}\perp h) & \text{ if } n_{i}\equiv 2 \mod 8 \,\,(\text{resp.}\,\, n_{i}\equiv 6 \mod 8)\end{cases}
\]
for some hyperbolic forms $h$, where $\phi_{i}$ denotes the quadratic form defined in (\ref{phiR})

Since $T^{-}_{\gamma_{i1}\tens \gamma_{i2}}=T_{\gamma_{i1}}\perp T_{\gamma_{i2}}$ in the Witt ring of $L'$, $T^{+}_{\gamma_{i1}\tens \gamma_{i2}}\equiv T_{\gamma_{i1}}\perp T_{\gamma_{i2}} \mod I^{4}(L')$ (see \cite[Example 11.3]{KMRT}), and $\operatorname{Ker}(\boldsymbol{\mathrm{e}}_{3})=I^{4}(L')$, where $I^{4}(L')$ denotes the $4$th power of the fundamental ideal, we have $\boldsymbol{\mathrm{e}}_{3}(\phi[r])(\eta)=\boldsymbol{\mathrm{e}}_{3}(\theta[r])(\eta')$, thus $\partial_{z}\big(\boldsymbol{\mathrm{e}}_{3}(\phi[r])(\eta)\big)\neq 0$, which completes the proof.\end{proof}

Finally, we prove the second main result on the group of unramified degree $3$ invariants for type $D$.

\begin{theorem}\label{secthmD}
	Let $G=(\prod_{i=1}^{m}\gSpin_{2n_{i}})/\gmu$ defined over an algebraically closed field $F$, $m\geq 1$, $n_{i}\geq 3$, where $\gmu$ is a central subgroup. Then, every unramified degree $3$ invariant of $G$ is trivial, i.e., $\Inv^{3}_{\nr}(G)=0$.
\end{theorem}
\begin{proof}
	Let $G_{\red}=(\prod_{i=1}^{m}\operatorname{\mathbf{\Omega}}_{2n_{i}})/\gmu$. As $BG$ is stably birational to $BG_{\red}$, it suffices to show that every nontrivial invariant of $G_{\red}$ is ramified. Let $\alpha$ be a normalized invariant in $\Inv^{3}(G_{\red})$ be written as in (\ref{DinvariantformD}) for some subsets $I_{1}'\subseteq I_{1}$, $I_{2}'\subseteq I_{2}$ and $R''\subseteq R'$.

	First, assume that there exists $j\in I'_{1}$. Let $Q=(x, y)$ be a division quaternion algebra over a field extension $K/F$ and let $\psi_{j}=\langle\langle x, y, z\rangle\rangle\perp h$ be a quadratic form over $E:=K((z))$, where $h$ denotes a hyperbolic form. Choose a $G_{\red}$-torsor $\eta=\big((A_{1},\sigma_{1}, f_{1}),\ldots, (A_{m},\sigma_{m}, f_{m})\big)$ such that
	\[(A_{j},\sigma_{j})=(M_{2n_{j}}(E), \sigma_{\psi_{j}}) \text{ and }  (A_{i},\sigma_{i})=(M_{2n_{i}}(E), t)   \]
	for all $1\leq i\neq j\leq m$, where $\sigma_{\psi_{j}}$ denotes the adjoint involution on $M_{2n_{j}}(E)$ with respect to $\psi_{j}$ and $t$ denotes the transpose involution on $M_{2n_{i}}(E)$. Then, we have
	\[\sum_{i\in I_{1}'}\boldsymbol{\mathrm{e}}_{3, i}(\eta)=(x, y, z),\,\,  \sum_{i\in I_{2}'}\Delta_{i}'(\eta)=\sum_{r\in R''}\boldsymbol{\mathrm{e}}_{3}(\phi[r])(\eta)=0.\]
	Therefore, we have $\partial_{z}(\alpha(\eta))=(x, y)\neq 0$. Hence, the invariant $\alpha$ ramifies.

	We assume that $I_{1}'=\emptyset$ and $I_{2}'\neq \emptyset$, i.e., $\alpha(\eta)=\sum_{i\in I_{2}'}\Delta_{i}'+\sum_{r\in R''}\boldsymbol{\mathrm{e}}_{3}(\phi[r])$. Let $j\in I_{2}'$ and let $\eta=\big((A_{1},\sigma_{1}, f_{1}),\ldots, (A_{m},\sigma_{m}, f_{m})\big)$ be a $G_{\red}$-torsor over $L$ as in the proof of Proposition \ref{formcorD}. Then, we have either
	\[\partial_{x_{j,1}}\big(\alpha(\eta)\big)=\partial_{x_{j,1}}\big(\Delta_{j}'(\eta)\big)\neq 0 \text{ or } \partial_{x_{j,2}}\big(\alpha(\eta)\big)=\partial_{x_{j,2}}\big(\Delta_{j}'(\eta)\big)\neq 0,\]
	thus the invariant $\alpha$ ramifies.

	Now we may assume that $I_{1}'=I_{2}'=\emptyset$. Then, the invariant $\alpha$ is ramified as shown in the last part of the proof of Proposition \ref{formcorD}.\end{proof}

\end{document}